\documentclass[letterpaper]{article}
\usepackage[hyperfootnotes=false, colorlinks, linkcolor={blue}, citecolor={magenta}, filecolor={blue}, urlcolor={blue}]{hyperref}

\usepackage{verbatim}
\usepackage{amssymb}
\usepackage{amscd}
\usepackage{amsmath}
\usepackage{amsthm}
\usepackage{eucal}
\usepackage{setspace}
\usepackage{url}
\usepackage[utf8]{inputenc}
\usepackage{booktabs}
\usepackage{enumerate}
\usepackage{enumitem}

\usepackage[T1]{fontenc}

\allowdisplaybreaks
\renewcommand{\a}{{\mathfrak a}}

\newcommand{\A}{{\mathbb A}}

\newcommand{\Q}{{\mathbb Q}}
\newcommand{\Z}{{\mathbb Z}}

\newcommand{\R}{{\mathbb R}}

\newcommand{\C}{{\mathbb C}}

\newcommand{\GL}{{\rm GL}}

\newcommand{\SL}{{\rm SL}}
\newcommand{\SO}{{\rm SO}}

\newcommand{\GSp}{{\rm GSp}}
\newcommand{\Sp}{{\rm Sp}}

\newcommand{\Sym}{\mathrm{Sym}}

\newcommand{\ii}{\sqrt{-1}}

\newcommand{\cG}{{\cal G}}
\newcommand{\cH}{{\cal H}}
\newcommand{\cP}{{\cal P}}
\newcommand{\cM}{{\cal M}}
\newcommand{\cN}{{\cal N}}
\newcommand{\cL}{{\cal L}}

\newcommand{\OO}{{\rm O}}

\newcommand{\trace}{{\rm tr}}

\newcommand{\Hom}{{\rm Hom}}

\renewcommand{\emptyset}{\varnothing}

\newcommand{\mat}[4]{{\setlength{\arraycolsep}{0.5mm}\left[
\begin{smallmatrix}#1&#2\\#3&#4\end{smallmatrix}\right]}}
\newcommand{\forget}[1]{}

\def\qdots{\mathinner{\mkern1mu\raise0pt\vbox{\kern7pt\hbox{.}}\mkern2mu
\raise3.4pt\hbox{.}\mkern2mu\raise7pt\hbox{.}\mkern1mu}}

\newtheorem{lemma}{Lemma.}[section]
\newtheorem{theorem}[lemma]{Theorem.}

\newtheorem{corollary}[lemma]{Corollary.}
\newtheorem{proposition}[lemma]{Proposition.}

\theoremstyle{remark}
\newtheorem{remark}{Remark}
\begin{document}
\begin{title}
{\Large\bf Petersson norms of Borcherds theta lifts to $\OO(1, 8n+1)$ with applications to injectivity and sup-norm bounds}
\end{title}
\author{Simon Marshall, Hiro-aki Narita and Ameya Pitale}
\maketitle
\begin{abstract}
We give an explicit formula for the Petersson norms of theta lifts from Maass cusp forms of level one to cusp forms on orthogonal groups $\OO(1,8n+1)$. Our formula explicitly determines  archimedean local factors of the norms.  As an application, we obtain the injectivity of the lifting of Maass forms and bounds on the sup-norm of cusp forms on these orthogonal groups in terms of their Laplace eigenvalues.
\end{abstract}

\tableofcontents

\section{Introduction}
In this paper, we consider several questions regarding lifts from Maass cusp forms $f$ of level $1$ to cusp forms $F_f$ on the orthogonal group $\OO(1,8n+1)$ with respect to the arithmetic subgroup $\Gamma$ attached to even unimodular lattices $L$ of rank $N=8n$. The group $\Gamma$, denoted by $\Gamma_S$ in (\ref{eq:GammaS}), is the subgroup of $\OO(1,8n+1)$ stabilizing $L\oplus\Z^2$ with $\Z^2$ viewed as a lattice of the hyperbolic plane. These lifts were constructed in \cite{LNP} by giving an explicit formula for the Fourier coefficients of $F_f$ and using the explicit theta lift construction due to Borcherds \cite{Bo} to prove the automorphy (see Section \ref{Autom-form}). In \cite{LNP}, it was shown that the map $f \to F_f$ preserves cuspidality and is Hecke equivariant. When restricted to Hecke eigenforms, it was shown that if the lift is non-zero, then it  provides counterexamples to the generalized Ramanujan conjecture (see Theorem \ref{LNP-theorem} below for details).

\subsection*{Petersson norm of the lift}

One of the main results of this paper is to obtain an explicit formula for the Petersson norm of the lift $F_f$. For this we use the well-known methods of Rallis \cite{Ra} of doubling integrals and the Siegel Weil formula. It should be remarked that Kudla-Rallis \cite{KR88-1} suggested the regularized Siegel Weil formula, which enables us to calculate the Petersson norm without the limitation of the original Siegel-Weil formula~(cf.~\cite{We2}), and Gan-Qiu-Takeda \cite{GQT} established the formula in full generality for dual pairs of unitary groups or symplectic-orthogonal groups. We assume that $f$ is a Hecke eigenform. To use the method of \cite{Ra}, we first obtain an adelization of the theta lift construction of Borcherds for our case. For this, we follow the work of Kudla \cite{Ku03} where the adelization of the Borcherds construction is obtained for orthogonal groups of signature $(p,2)$. The $L^2$-norm of the adelization $\Phi$ of $F_f$ can then be rewritten as 

$$||\Phi||^2 = \int\limits_{\SL_2(\Q) \backslash \SL_2(\A)} \Big(\int\limits_{\SL_2(\Q) \backslash \SL_2(\A)} f_0(g_1) E((g_1^*, g_2), s_0; \Xi_0) dg_1 \Big) \overline{f_0(g_2)} dg_2.$$
Here, $f_0$ is the adelization of $f$, $\Xi_0$ is a section in an induced representation of $\Sp_4(\A)$ obtained from the Weil representation corresponding to the theta lifting, and $E$ is an Eisenstein series on $\Sp_4$ obtained from the Siegel Weil formula. As in \cite{PSS20}, the inner integral is Eulerian and can be written as a product of local integrals~(see also \cite[Section 11]{GQT}). Since we are restricted to Maass forms $f$ of level $1$, in the non-archimedean case, all the data is unramified and the integral is obtained in \cite{PSS20}. The main contribution here is the archimedean integral computation, which is never an immediate consequence from general formulas cited above. The key ingredient of the archimedean computation is to realize (see Proposition \ref{arch-K-inv-prop}) that the archimedean section $\Xi_\infty$ is in the trivial $K$-type of the $\Sp_4(\R)$ representation, i.e. it is invariant under the maximal compact subgroup of $\Sp_4(\R)$.

In Section \ref{arch-comp-sec}, we compute the local archimedean integral to obtain the following result (see Theorem \ref{main-norm-thm}).

\begin{theorem}\label{norm-thm-intro}
Let $L$ be an even unimodular lattice of dimension $N$. Let $f \in S(\SL_2(\Z);-\frac{r^2 + 1}{4})$ be a Hecke Maass eigenform with respect to $\SL_2(\Z)$ and let $\pi$ be the irreducible cuspidal automorphic representation of $\GL_2(\A)$ corresponding to $f$.  Let $F_f$ be the lift of $f$ to a cusp form on $\OO(1, N+1)$ with respect to the arithmetic subgroup $\Gamma$ attached to $L$. Then the Petersson norm of $F_f$ is given by the formula
$$||F_f||^2 = \frac{L(\frac N2, \pi, {\rm Ad})}{\zeta(\frac N2 + 1)\zeta(N)} \Big(2^{1-\frac{N}{2}}\pi^2 \frac{\Gamma(\frac N4+\frac{\sqrt{-1}r}2) \Gamma(\frac N4-\frac{\sqrt{-1}r}2)}{\Gamma(\frac N4 + \frac 12)^2}\Big) ||f||^2.$$
Here, $L(s, \pi, {\rm Ad})$ is the finite part of the degree $3$ adjoint $L$-function of $\pi$. 
\end{theorem}

Let us remark here that, in the case of signature $(p,2)$, the Petersson inner product of the Borcherds lift (or the Kudla-Millson lift) has been computed in Theorem 4.9 of \cite{Br-Fu}. In that case, the authors start with a holomorphic modular form of full level and the archimedean integral then corresponds to a computation involving the holomorphic discrete series of $\GL_2$. 

In \cite{Mu-N-P} and \cite{NPW}, we constructed the Borcherds lifting from Maass forms with square free level $m$ to cusp forms on $\OO(1,5)$ with respect to arithmetic subgroups corresponding to maximal orders in definite quaternion algebras of discriminant $m$. The adelization of the lift and the Rallis inner product method can be applied to this case as well. The local archimedean integral can be computed similar to the computation in this paper. The main difference is the computation of the local ramified integral in the case $p | m$ which involves vector valued modular forms. We will be working on this case in the future.

\subsection*{Injectivity of the lifting map}
As is the case with most theta lift constructions, proving the automorphy of the lift is straightforward, whereas showing non-vanishing is often highly non-trivial. In \cite{LNP}, using the Fourier expansion of $F_f$ at the cusp at infinity, we were able to obtain the injectivity of the map $f \to F_f$ 
in the case that the norm map on the lattice $L$ is surjective. 


For general even unimodular lattices, an immediate corollary to Theorem \ref{norm-thm-intro} is the injectivity of the map $f \to F_f$, when restricted to Hecke eigenforms. We can extend this injectivity to non-Hecke eigenforms by using the Linear Algebra trick used in \cite{NPW} (see Theorem 7.1 of \cite{NPW}). Note that the latter requires information on Hecke eigenvalues of $F_f$ when $f$ is a Hecke eigenform, and these were computed in \cite{LNP}. See Corollary \ref{inj-cor} for details. Note that one could possibly extend the injectivity of the Kudla-Millson lift obtained in \cite{Br-Fu} to non-Hecke eigenforms if one had further information on the  lift in the case of Hecke eigenforms.

Let us remark here that it is well known that all even unimodular lattices of dimension $N$ form a single genus, and are in one-to-one correspondence with the $\Gamma$-cusps (see Section \ref{Four-exp-sec}). If we considered the Fourier expansion of $F_f$ at a $\Gamma$-cusp corresponding to an even unimodular lattice whose image of the norm map exhausts all positive integer, then it is possible to follow the method in \cite{LNP} to get injectivity of the map $f \to F_f$ in general.

\subsection*{Bounds on sup-norm of the lift $F_f$}

Another main result of this paper is to prove bounds on the sup-norm of $F_f$.  The problem of estimating the sup-norm of Laplace eigenfunctions on a compact Riemannian manifold is a fundamental one in harmonic analysis and mathematical physics.  The most basic result on this problem is due to Avacumovi\'{c} \cite{Av} and Levitan \cite{Lev}.  If $X$ is a compact Riemannian manifold without boundary and $\phi$ is an eigenfunction of the Laplacian on $X$ with eigenvalue $-\Lambda \le 0$, they show that 
\begin{equation}\label{SS-eqn}
\frac{||\phi||_\infty}{||\phi||
_2} \ll \Lambda^{\frac {{\rm dim} X-1}4}.
\end{equation}
In general, this is the best possible estimate and is indeed achieved when $X$ is a sphere.  On the other hand, one expects stronger bounds to hold on a generic manifold.  For instance, when $X$ is negatively curved, B\'{e}rard \cite{Be} showed that the bound (\ref{SS-eqn}) can be improved by a factor of $\sqrt{\log \Lambda}$.  Moreover, it is often possible to improve (\ref{SS-eqn}) by a power of $\Lambda$ under arithmetic assumptions on $X$ and $\phi$.  The first such improvement is due to Iwaniec and Sarnak \cite{Iw-Sa-2}. They considered $X$ a congruence arithmetic hyperbolic surface arising from a quaternion algebra over $\Q$ (possibly split), and $\phi$ a Hecke--Maass form.  In this case, they improved on (\ref{SS-eqn}) to obtain 
\begin{equation}\label{Iw-Sa-bound}
\frac{||\phi||_\infty}{||\phi||
_2} \ll_\epsilon \Lambda^{\frac 5{24}+\epsilon}.
\end{equation}
They also obtained a lower bound $\sqrt{ \log \log \Lambda}$ for an infinite sequence of $\phi$, which was later improved to $\exp( (1 + o(1)) \sqrt{ \log \Lambda / \log \log \Lambda})$ by Mili\'{c}evi\'{c} \cite{Mil2}. Let us also mention several other works on the sup-norm problem \cite{Bl-Ho}, \cite{Bl-P}, \cite{Bru-Ma}, \cite{Bru-Te}, \cite{Sa17}, \cite{Sa172}. Among other things we also cite \cite{Ra-Wa}, which includes an excellent review of the sup-norm problems in terms of elliptic differential operators.

We note that in these arithmetic settings, where $X$ is taken to be a locally symmetric space, one generally assumes that $\phi$ is an eigenfunction not just of the Laplacian, but of the full ring of invariant differential operators.  Under this assumption on $\phi$, it was shown by Sarnak \cite{Sa} that the bound (\ref{SS-eqn}) can be strengthened to
\begin{equation}\label{compact-bound}
\frac{||\phi||_\infty}{||\phi||
_2} \ll \Lambda^{\frac{{\rm dim} X - {\rm Rank} X}4}.
\end{equation}
This is sharp on spaces of compact type, and is the natural analog of (\ref{SS-eqn}) for these eigenfunctions.
\medskip

It is also interesting to investigate the sup-norm of eigenfunctions in the case when $X$ is not compact.  In this case, the bound (\ref{SS-eqn}) (and (\ref{compact-bound}), under the appropriate assumptions) continues to hold on fixed compact sets, and it is natural to ask whether it in fact holds globally (i.e. over the entire manifold).  For this to happen one certainly needs $\| \phi \|_\infty$ to be finite, and this is not generally true, even when $X$ is a finite volume locally symmetric space.  However, it is true if one assumes that $X$ is a finite volume locally symmetric space and $\phi$ is cuspidal~(cf.~\cite{HC}), and so we make these assumptions from now on.

In some cases, it is known that the bounds (\ref{SS-eqn}) and (\ref{compact-bound}) hold globally.  See for instance \cite{BHMM} for the case of Hecke--Maass forms on ${\rm GL}_2$ over a number field, and \cite{BHM} for Hecke--Maass forms on the space ${\rm PGL}_3(\Z) \backslash {\rm PGL}_3(\R) / {\rm PO(3)}$.  (Note that \cite{BHM} establishes a bound stronger than (\ref{SS-eqn}), but not as strong as (\ref{compact-bound}).)  On the other hand, it was shown in \cite{Bru-Te} that (\ref{SS-eqn}) fails on ${\rm PGL}_n(\Z) \backslash {\rm PGL}_n(\R) / {\rm PO(n)}$ for $n \ge 6$.  The reason for this failure is the large peaks of the ${\rm GL}_n$ Whittaker function in higher rank, which lead to large values of cusp forms via the Fourier expansion.  Moreover, it is generally expected that the large values produced in this way occur high in the cusp, at height roughly $\sqrt{\Lambda}$; \cite{Bru-Te} establishes the weaker result that the suprema of these cusp forms occur at points tending to infinity.  This phenomenon of cusp forms having a peak high in the cusp is a general one, which is already present for ${\rm GL}_2$ where it is caused by the transition range of the Bessel function $K_{\ii r}(y)$ at $y \sim r$ \cite{Sa}.  However, in this case the peak produced is only of size $\Lambda^{1/12}$, which is smaller than (\ref{SS-eqn}).  These results lead one to ask the general question of whether a sequence of cusp forms on $X$ realize their suprema in a fixed compact set, or at a sequence of points tending to infinity.

We remark that if, instead of Maass forms, one considers holomorphic modular forms of growing weight for the group ${\rm SL}_2(\Z)$, then the growth rate of the sup norm was determined by Xia \cite{Xia} using the Fourier expansion.  Similar results in the case of Siegel modular forms were obtained by Blomer \cite{Bl}.
\medskip

We prove the following result on the sup norm of the forms $F_f$.

\begin{theorem}\label{supnorm-thm-intro}
Let $f \in S(\SL_2(\Z);-\frac{r^2 + 1}{4})$ be a non-zero Hecke Maass eigenform with Fourier coefficients $\{c(m) : m \in \Z\}$ satisfying $c(m) = \pm c(-m)$ for all $m \in \Z$. Suppose $L$ is an even unimodular lattice of dimension $N$. Let $F_f$ be the lift of $f$ to a cusp form on $\OO(1, N+1)$ with respect to the arithmetic subgroup $\Gamma$ attached to $L$. Let $-\Lambda$ be the eigenvalue of $F_f$ with respect to the Casimir operator.  Then for any $\epsilon > 0$ and any $r$, we have
$$\frac{||F_f||_\infty}{||F_f||_2} \ll_{N, \epsilon, \Gamma} \Lambda^{\frac N4+\frac{N(1+2\theta)}{8(N+1+2\theta)} + \epsilon} \leq   \Lambda^{\frac{N}4 + \frac{\theta}4+ \frac 18+ \epsilon},$$
where $\theta = 7/64$ is the current best estimate towards the Ramanujan conjecture for Maass forms. We also have the lower bound 
\begin{equation}\label{supnorm-thm-lower}
\Lambda^{\frac{N}8 + \frac{1}{12} - \epsilon} \ll_{N, \epsilon, \Gamma} \frac{||F_f||_\infty}{||F_f||_2}.
\end{equation}
\end{theorem}

See Theorems \ref{Upper-bd-thm} and \ref{trace-thm} for more details and more general results.  We note that the bound (\ref{SS-eqn}) in this case reads $||F_f||_\infty / ||F_f||_2 \ll \Lambda^{N/4}$, so that we are not quite able to obtain this globally.  Note that we are free to assume that $r \gg 1$.

\begin{remark}

In contrast to the lower bounds we prove here, there are several papers that establish power growth of eigenfunctions on hyperbolic manifolds, and more generally on locally symmetric spaces of noncompact type, in {\it fixed} compact sets \cite{Bru-Ma, BrumleyMar, Do, LapidOffen, Milicevic, Ru-Sa}.  We mention in particular \cite{BrumleyMar, Do}, which apply to the higher-dimensional hyperbolic setting considered here.  (The results of \cite{Bru-Ma} also include this, but the growth exponents produced are ineffective.)  In \cite{Do}, Donnelly constructs compact hyperbolic $(N+1)$-manifolds for $N \ge 4$, and sequences of Laplace eigenfunctions $\{ \phi_i \}_i$ that satisfy $\| \phi_i \|_\infty / \| \phi_i \|_2 \gg \Lambda_i^{(N-3)/4}$.  In \cite{BrumleyMar}, this lower bound was improved to $\Lambda_i^{(N-1)/4 - \epsilon}$ for $N$ even, and it is expected that this parity condition can be removed.  Moreover, similarly to $F_f$, the forms constructed in these papers are theta lifts from ${\rm SL}_2$.  These results lead one to hope that $F_f$ might satisfy the same lower bound, which should be realized in a fixed compact subset of the manifold.  As this is larger than the lower bound of $\Lambda^{\frac{N}8 + \frac{1}{12} - \epsilon}$ obtained from the peak of the Whittaker function, one might therefore expect $F_f$ to realize their sup norms in the bulk, rather than the cusp.

\end{remark}

We now give an outline of the proof of Theorem \ref{supnorm-thm-intro}.  The proof works by combining two ingredients.  The first is a standard upper bound coming from the pre-trace formula, which is valid for any square-integrable Laplace eigenfunction on any finite-volume hyperbolic orbifold.  This is Theorem \ref{trace-thm}, which gives
\begin{equation}\label{PTF-bd-intro}
\frac{|F_f(n(x)a_y)|}{||F_f||_2} \ll \Lambda^{\frac N4} + y^{\frac N2} \Lambda ^{\frac N8}.
\end{equation}
Here, $n(x)a_y$ are Iwasawa coordinates~(cf.~(\ref{Iwasawa-decomp})) adapted to one of the cusps.

The second ingredient is upper and lower bounds proved using the Fourier expansion of $F_f$. These bounds rely on the fact that these forms are theta lifts in two ways; first, by providing an explicit formula for the Fourier coefficients of $F_f$ in terms of those of $f$, and secondly through the formula for $\| F_f \|_2$ from Theorem \ref{norm-thm-intro}.  The lower bound comes from the first nonzero Fourier coefficients of $F_f$, combined with the transition behaviour of the Bessel function; see the end of Section \ref{Four-exp-sec}.

The upper bound is stated in Theorem \ref{Upper-bd-thm}, and gives
\begin{equation}\label{FE-bd-intro}
\frac{1}{ \| F_f \|_2 } | F_f( n(x) a_y) | \ll_{\epsilon, N, L} 
\begin{cases}
y^{-N/2 -1 - 2\theta} r^{3N/4 + 1 +2\theta + \epsilon} & 1 \ll y < r^{11/12}; \\
y^{-N/2 +1 - 2\theta} r^{3N/4 - 5/6 + 2\theta + \epsilon} & r^{11/12} < y \le r/2\pi; \\
e^{-Cy} & r / 2\pi < y. 
\end{cases}
\end{equation}
Here, $\theta = 7/64$ is the current best estimate towards the Ramanujan conjecture for Maass forms, and we note that $\Lambda \sim r^2$.  We combine (\ref{PTF-bd-intro}) and (\ref{FE-bd-intro}) by noting that the first bound is strong low in the cusp, while the second is strong high in the cusp.  In fact, the first bound gives $\frac{||F_f||_\infty}{||F_f||_2} \ll \Lambda^{N/4}$ when $y < \Lambda^{1/4}$, while for $y > \Lambda^{1/4}$ the second gives $\frac{||F_f||_\infty}{||F_f||_2} \ll \Lambda^{N/4 + \frac14 + \frac \theta{2}}$.  Finding the point at which they are equally strong gives the upper bound of Theorem \ref{supnorm-thm-intro}.

\subsection*{Bounds on Fourier coefficients of the lift $F_f$}

As discussed above, Theorems \ref{supnorm-thm-intro} and \ref{Upper-bd-thm} on sup-norm bounds rely on an estimate for the Fourier coefficients of $F_f$. By using standard techniques, one can obtain the Hecke bound for the Fourier coefficients $\{A(\lambda) : \lambda \in L \backslash \{0\}\}$  of $F_f$ given by $|A(\lambda)| \ll |\lambda|^{\frac N2}$. Here, the implied constant depends on the sup-norm of $F_f$, the quantity we wish to bound. Hence, the Hecke bound is not adequate for obtaining the sup-norm bounds on $F_f$. Using the explicit formula for $A(\lambda)$ in terms of the Fourier coefficients of the Maass form $f$, we are able to obtain in Proposition \ref{Petersson-norm-estimate} the improved bounds of the form $|A(\lambda)| \ll |\lambda|^{\frac N2-1+\epsilon}$. In the special case of primitive $\lambda$, i.e., satisfying $\frac 1n \lambda \notin L$ for all  integers $n > 1$, we get an even better bound $|A(\lambda)| \ll |\lambda|^{2 \theta+1+\epsilon}$, where $\theta = 7/64$ is the current best estimate towards the Ramanujan conjecture for Maass forms. Here the implied constants depend on $\epsilon, r$ and $||f||_2$.

These bounds on the size of the Fourier coefficients of the lift are analogous to the results obtained by Ikeda and Katsurada \cite{IK} in the context of Siegel cusp forms. The Hecke bound for a Siegel cusp form of genus $n$ and weight $k$ is given by $|A(T)| \ll (\det(2T))^{\frac k2}$, where  $\{A(T)\}$ are the  Fourier coefficients with $T$ running through  $n \times n$ half integral, positive definite, symmetric matrices. In \cite{IK},  for the subset of Ikeda lifts, the authors prove better bounds given by $|A(T)| \ll (\det(2T))^{\frac k2-\frac 12}$. Furthermore, they show that if $T$ is primitive, then we have $|A(T)| \ll (\det(2T))^{\frac k2 - \frac 1{12} - \frac n4}$.

\subsection*{Outline of the paper}

The paper is outlined as follows. In Section \ref{Autforms-sec}, after reviewing basic notion on orthogonal groups, their algebraic subgroups and real hyperbolic spaces, we introduce automorphic forms that are the focus of this paper. We recall the construction of the theta lifting map as worked out in \cite{LNP} and state the main theorem from that work regarding these lifts. In Section \ref{adelic-Borcherdslift}, we carry out the adelization of the theta lifts. In Section \ref{PeterssonNormsec}, we set up the global integrals for the Petersson norm of the lifting using the method of Rallis. The bulk of the computation is done in Section \ref{arch-comp-sec}, where the archimedean integral is calculated. We end Section \ref{arch-comp-sec} with the explicit formula for the Petersson norm of the lift and the injectivity of the lifting map. In Section \ref{Four-exp-sec}, we compute upper and lower bounds on the sup-norm of the lift using its Fourier expansion. In Section \ref{Tr-formula-sec}, we apply the pre-trace formula method to obtain upper bounds for any $L^2$ eigenfunction of the Laplacian on an $(N+1)$-dimensional hyperbolic orbifold. We put the results from these methods together to obtain the proof of Theorem \ref{supnorm-thm-intro} in Section \ref{thm-pf-sec}.

\subsection*{Acknowledgement}

The authors would like to thank Jens Funke for discussions regarding the local computations of the Petersson norm. We would like to thank Abhishek Saha for  giving expert advice on the analytic number theoretic aspect of the paper.  The first author would like to thank Victoria University of Wellington for their hospitality while this paper was being completed.  The first author was supported by NSF grant DMS-1902173.

\section{Classical automorphic forms}\label{Autforms-sec}
\subsection{Algebraic groups}\label{gps-hypsp}
For $N \in \mathbb{N}$, let $S\in M_N(\Q)$ be a positive definite symmetric matrix and put $Q:=
\begin{bmatrix}
& & 1\\
& -S & \\
1 & &
\end{bmatrix}$. 
We then define a $\Q$-algebraic group $\cG$ by the group 
\[
\cG(\Q):=\{g\in M_{N+2}(\Q)\mid {}^tgQg=Q\}
\]
of $\Q$-rational points. We introduce another $\Q$-algebraic group $\cH$ by the group 
\[
\cH(\Q):=\{h\in M_N(\Q)\mid {}^thSh=S\}
\]
of $\Q$-rational points.
Let $q_S$, resp.\ $q_Q$, denote the quadratic form on $\Q^N$, resp.\ $\Q^{N+2}$, associated to $S$, resp.\ $Q$, i.e.
$$
q_S(v) = \frac{1}{2}{}^tv S v,~
q_Q(w) = \frac{1}{2}{}^tw Q w
$$
for $v \in \Q^N$ and $w \in \Q^{N+2}$.
Then $\cH$, resp.\ $\cG$, is the orthogonal group associated to this quadratic form $q_S$, resp. $q_Q$.
For every place $v\leq\infty$ of $\Q$ we put $G_v:=\cG(\Q_v)$ and  $H_v:=\cH(\Q_v)$.

In addition, we introduce the standard proper $\Q$-parabolic subgroup $\cP$ of $\cG$ with the Levi decomposition $\cP=\cN\cL$, 
where the $\Q$-subgroups $\cN$ and $\cL$ are defined by
\begin{align*}
\cN(\Q)&:=\left\{\left.n(x)=
\begin{bmatrix}
1 & {}^txS & \frac{1}{2}{}^txSx\\
  & 1_N & x\\
  &      & 1
\end{bmatrix}~\right|~x\in\Q^N\right\}, \\
\cL(\Q)&:=\left\{\left.
\begin{bmatrix}
\alpha & & \\
  & \delta & \\
  &          & \alpha^{-1}
\end{bmatrix}~\right|~\alpha\in\Q^{\times},~\delta\in\cH(\Q)\right\}.  
\end{align*}

Let us regard $J = \Z^2$ as a lattice of the real hyperbolic plane, and let $L$ be a maximal lattice with respect to $S$. We then put
\[
L_0:=\left\{\left.\begin{bmatrix}
x\\
y\\
z
\end{bmatrix}\in\Q^{N+2}\right|~x,z\in\Z,~y\in L\right\} = L \oplus J,
\]
which is a maximal lattice with respect to $Q$. Here, see \cite[Chapter II,~Section 6.1]{Shi} for the definition of maximal lattices.
Through the bilinear form induced by the quadratic form $q_S$, the dual lattice $L^\sharp := \Hom_\Z(L, \Z)$ is identified with a sublattice of $\Q^N$ containing $L$.
\vskip 0.2in

{\bf Assumption: We will assume that $L$ is an unimodular even lattice. This implies that $8 | N$ and $L^\sharp = L$.}

\vskip 0.2in

For each finite prime $p<\infty$ we introduce $L_{0,p}:=L_0\otimes_{\Z}\Z_p$ and put
\[
K_p:=\{g\in G_p\mid gL_{0,p}=L_{0,p}\},
\]
which forms a maximal open compact subgroup of $G_p$. 
On the other hand, let $R:=
\begin{bmatrix}
1 & & \\
& S & \\
& & 1
\end{bmatrix}$ and put
\[
K_{\infty}:=\{g\in G_{\infty}\mid {}^tgRg=R\}, 
\]
which is a maximal compact subgroup of $G_{\infty}$. 
Let $K_f:=\prod_{p<\infty}K_p$ and $K:=K_f\times K_{\infty}$. The groups $K_f$ and $K$ form maximal compact subgroups of $\cG(\A_f)$ and $\cG(\A)$ respectively.
We furthermore put $U:=U_f\times H_{\infty}$ with $U_f:=\prod_{p<\infty}U_p$, where
\[
U_p:=\{h\in H_p\mid hL_p=L_p\}
\]
with $L_p:=L\otimes_{\Z}\Z_p$. We now set
\begin{equation}
  \label{eq:GammaS}
\Gamma_S:=\cG(\Q)\cap K_f G_{\infty} = \{\gamma\in\cG(\Q)\mid \gamma L_0=L_0\}. 
\end{equation}
We have the following result \cite[Lemma 2.1]{LNP}, 
\begin{lemma}\label{Classnum-Cusps}
\begin{enumerate}
\item (Strong approximation theorem for $\cG$)~The class number of $\cG=\OO(Q)$ with respect to $G_{\infty}K_f$ is one. Namely $\cG(\A)=\cG(\Q)G_{\infty}K_f$.
\item The class number of $\cH=\OO(S)$ with respect to $U$ coincides  with the number of $\Gamma_S$-cusps. 
\end{enumerate}
\end{lemma}

The real Lie group $G_{\infty}$ admits an Iwasawa decomposition 
\[
G_{\infty}=N_{\infty}A_{\infty}K_{\infty},
\]
where
\begin{equation}\label{Iwasawa-decomp}
N_{\infty}:=\left\{n({ x})\mid { x}\in\R^N\right\}, \qquad 
A_{\infty}:=\left\{\left.
a_y=
\begin{bmatrix}
y &  & \\
 & 1_N & \\
 & & y^{-1}
\end{bmatrix}~\right|~y\in\R^{\times}_+\right\}.
\end{equation}
From the Iwasawa decomposition we can identify the homogeneous space $G_{\infty}/K_{\infty}$ with the $(N+1)$-dimensional real hyperbolic space $H_N:=\{(x,y)\mid x\in\R^N,~y\in\R_{>0}\}$ by the natural map
\[
n(x)a_y\mapsto (x,y).
\]
The cusp forms we are going to study are regarded as cusp forms on the real hyperbolic space $H_N$.

\subsection{Automorphic forms and lifting theorem}\label{Autom-form}
For $\lambda\in\C$ and a congruence subgroup $\Gamma\subset \SL_2(\R)$ we denote by $S(\Gamma,\lambda)$ the space of Maass cusp forms of weight $0$ on the complex upper half plane ${\mathfrak h}:=\{u+\sqrt{-1}v\in\C\mid v>0\}$ whose eigenvalue with respect to the hyperbolic Laplacian is $-\lambda$.

 For $r\in\C$ we denote by $\cM(\Gamma_S,r)$ the space of smooth functions $F$ on $G_{\infty}$ satisfying the following conditions: 
\begin{enumerate}
\item $\Omega\cdot F=\displaystyle\frac{1}{2N}\left(r^2-\displaystyle\frac{N^2}{4}\right)F$, where $\Omega$ is the Casimir operator defined in \cite[(2.3)]{LNP},
\item for any $(\gamma,g,k)\in\Gamma_S\times G_{\infty}\times K_{\infty}$, we have $F(\gamma gk)=F(g)$,
\item $F$ is of moderate growth.
\end{enumerate}
Let $r \in \R$. We say that $F \in \cM(\Gamma_S, \sqrt{-1}r)$ is a cusp form if it vanishes at the cusps of $\Gamma_S$. By \cite[Section 2.3]{LNP}, any cusp form $F \in \cM(\Gamma_S, \sqrt{-1}r)$ has a Fourier expansion of the form
\begin{equation}\label{F-Fourier-exp}
F(n(x)a_y)=\sum_{\lambda\in L\setminus\{0\}}A(\lambda)y^{\frac N2}K_{\sqrt{-1}r}(4\pi|\lambda|_Sy)\exp(2\pi\sqrt{-1}{}^t\lambda Sx).
\end{equation}
Here, $|\lambda|_S:=\sqrt{q_S(\lambda)}$ and $K_{\sqrt{-1}r}$ is the $K$-Bessel function.

Let 
\begin{equation}\label{f-Four}
f(\tau)=\sum_{n\not=0}c(n)W_{0,\frac{\sqrt{-1}r}{2}}(4\pi|n|v)\exp(2\pi\sqrt{-1}nu)\in S(\SL_2(\Z);-\frac{r^2 + 1}{4})
\end{equation}
 be a Maass cusp form on ${\mathfrak h}$, where we use the Whittaker function $W_{0,\frac{\sqrt{-1}r}{2}}$ to describe the Fourier expansion of $f$. Recall that we have supposed that $L$ is an even unimodular lattice of rank $N$ divisible by $8$. For $0 \neq \lambda \in L$, define
\begin{equation}\label{Alambda-defn}
A(\lambda):=|\lambda|_S\sum_{d|d_{\lambda}}c\left(-\frac{|\lambda|_S^2}{d^2}\right)d^{\frac N2-2},
\end{equation}
where $d_{\lambda}$ denotes the greatest integer such that $\frac 1{d_\lambda} \lambda \in L$. By Theorems 3.1, 3.3, 4.11 and 5.6 from \cite{LNP} we have
\begin{theorem}\label{LNP-theorem}
Let $L$ be an even unimodular lattice. Let $f \in  S(\SL_2(\Z);-\frac{r^2 + 1}{4})$ with Fourier expansion (\ref{f-Four}). Let $F_f : H_N \to \C$ be given by the Fourier expansion (\ref{F-Fourier-exp}) with Fourier coefficients $A(\lambda)$ defined in (\ref{Alambda-defn}). Then 
\begin{enumerate}
\item The map $f \to F_f$ is a map from $S(\SL_2(\Z);-\frac{r^2 + 1}{4})$ to $ \cM(\Gamma_S,\sqrt{-1}r)$ preserving cuspidality.
\item If $f$ is a Hecke eigenform, then so is $F_f$.
\item Suppose $f$ is a Hecke eigenform. Let $\pi_{F_f}$ be the cuspidal automorphic representation of $\cG(\A)$ generated by $F_f$. Then $\pi_{F_f}$ is irreducible, and thus has the decomposition into the restricted tensor product $\otimes'_{v\le\infty}\pi_{F_f, v}$ of irreducible admissible representations $\pi_{F_f, v}$ of $G_v$. For $v=p<\infty$, the representation $\pi_{F_f, p}$ is the spherical constituent of an unramified principal series representation of $G_p$.
\item The finite part of the degree $N+2$ standard $L$-function of $\pi_{F_f}$ is given by
$$L(s, \pi_{F_f}) = L(s, \pi_f, {\rm Sym}^2) \prod\limits_{i=-(\frac N2-1)}^{\frac N2-1} \zeta(s-i),$$
where $\pi_f$ is the cuspidal, automorphic representation of $\GL_2(\A)$ generated by $f$ and $ L(s, \pi_f, {\rm Sym}^2)$ is the degree $3$ symmetric square $L$-function of $\pi_f$.
\item For every finite prime $p<\infty$, $\pi_{F_f, p}$ is non-tempered while $\pi_{F_f, \infty}$ is tempered.
\end{enumerate}
\end{theorem}
Let us remark here that, in \cite{LNP}, the injectivity of the map $f \to F_f$ was proven only for those lattices $L$, which have the property that it contains vectors of length $M$ for {\it all} $M \in \Z_{>0}$. In Corollary \ref{inj-cor} below, we prove the injectivity of the map $f \to F_f$ in full generality.

\section{Adelization of the Borcherds lift}\label{adelic-Borcherdslift}
In \cite{LNP}, the automorphy of the lift $F_f$ is proved by showing that it is a Borcherds theta lift. In \cite[Section 3]{LNP} the classical construction of the theta functions and Borcherds lifts has been explained in details. The first main result of this paper is the Petersson norm for $F_f$. We wish to use the Siegel-Weil formula and the Rallis inner product formula for this. For this we first need to adelize the Borcherds lift.  The main reference for this is \cite{Ku03}, where the adelic Borcherds lift has been worked out for signature $(p, 2)$. We will do this for signature $(1, N+1)$. Recall that we have assumed that   $L$ is an even unimodular lattice. From now on, let $V_N$ be the quadratic  space of dimension $N+2$ defined over $\Q$ equipped with the quadratic form  $q_Q$ defined in Section \ref{gps-hypsp}.  Let $B_Q$ be the bilinear form corresponding to the quadratic form  $q_Q$. By $V_N(\R)$ and $V_N(\A)$ we denote the real quadaratic space and  the adelic quadratic space attached to $V_N$ respectively. The former is often denoted simply by $\R^{1,N+1}$. By $V_N(\A_f)$ we denote the space formed by the finite adeles in $V_N(\A)$. 

Let $\mathcal D$ be the Grassmanian  of positive oriented lines in the real quadratic space $V_N(\R)$.  For $(x,y) \in H_N$, set $\nu(x, y) := \frac 1{\sqrt{2}} {}^t(y + y^{-1}q_S(x), -y^{-1}x, y^{-1}) \in V_N(\R)$ satisfying $B_Q(\nu(x, y), \nu(x, y)) = 1$. Then, we can identify $H_N$ with one of the two connected components $\mathcal D^+$ of $\mathcal D$ via
\begin{equation}\label{HN=D}
H_N \ni (x, y) \rightarrow \R \cdot \nu(x, y) \in \mathcal D^+.
\end{equation}

 Denote by $X = \cG(\Q) \backslash (\mathcal D \times \cG(\A_f))/K_f$. By Lemma \ref{Classnum-Cusps},  we have $\cG(\A) = \cG(\Q) \cG(\R) K_f$. We have $X \simeq \Gamma_S \backslash \mathcal D^+$, where $\Gamma_S = \cG(\Q) \cap \cG(\R)K_f$.

Let $S(V_N(\A)), S(V_N(\A_f))$ and $S(V_N(\R))$ be the space of Schwartz Bruhat  functions of $V_N(\A), V_N(\A_f)$ and $V_N(\R)$ respectively. For $\nu \in \mathcal D$, we have the map 
\begin{align*}
\iota_{\nu}:{V_N(\R)} & \to {\R \cdot \nu \oplus(\nu^{\perp},{q_S|_{{\nu}^{\perp}}})\simeq \R^{1,N+1}}\\
\lambda&\mapsto (\iota^+_\nu(\lambda), \iota^-_\nu(\lambda)).
\end{align*}
Here, $\iota^+_\nu$ and $\iota^-_{\nu}$ are the projections from $V_N(\R)$ to $\R \nu$ and $\nu^{\perp},{q_S|_{{\nu}^{\perp}}}$ respectively. For $\lambda \in V_N(\R)$, set $R(\lambda, \nu) := -2q_Q(\iota^+_\nu(\lambda))$ and $(\lambda, \lambda)_\nu := 2q_Q(\lambda) + 2R(\lambda, \nu)$. We can see that
\begin{align*}
(\lambda, \lambda)_\nu &= 2q_Q(\lambda) + 2R(\lambda, \nu) = 2q_Q(\iota^+_\nu(\lambda)) +
 2q_Q(\iota^-_\nu(\lambda)) - 4q_Q (\iota^+_\nu(\lambda)) \\
 &= 2q_Q(\iota^-_\nu(\lambda)) - 2q_Q(\iota^+_\nu(\lambda)) = 2q_Q(\lambda_{\nu^-}) - 2q_Q(\lambda_{\nu^+}).
 \end{align*}
Here we are denoting $\iota^+_\nu(\lambda) = \lambda_{\nu^+}$ and  $\iota^-_\nu(\lambda) = \lambda_{\nu^-}$. 
Let $A^\circ(\mathcal D)$ be the space of smooth functions on $\mathcal D$. 
We then introduce the Gaussian 
\begin{equation}\label{arch-gaussian-defn}
\tilde\phi_\infty(\lambda,\nu) := {\rm exp}(\pi (\lambda, \lambda)_\nu) = {\rm exp}\big(2 \pi q_Q(\lambda_{\nu^-}) - 2 \pi q_Q(\lambda_{\nu^+})\big).
\end{equation}
We can regard this as a $A^\circ(\mathcal D)$-valued Schwartz function on $V_{N}(\R)$ by
\[
V_{N}(\R)\ni\lambda\rightarrow ({\mathcal D}\ni\nu\mapsto\tilde\phi_\infty(\lambda, \nu)).
\]

Note that, in \cite{Ku03}, there is a minus sign in the exponent instead of plus sign. The reason for the difference is that we are considering signature $(1, N+1)$ instead of $(p, 2)$. Consider the polynomial $P(\lambda, \nu)$ obtained by applying the operator ${\rm exp}(-\Delta/(8 \pi))$ to the polynomial $2^{-\frac N4-3} x_1^{\frac N2}$
defined in \cite[Section 3.2]{LNP}, where $\Delta$ denotes the Laplacian for the $(N+2)$-dimensional Euclidean space with the coordinate $(x_1,\cdots,x_{N+2})$. For $ \lambda \in V_N(\R)$ and $\nu \in \mathcal D$, define
\begin{equation}\label{arch-phi}
\phi_\infty(\lambda, \nu) := P(\lambda, \nu) \tilde\phi_\infty(\lambda, \nu).
\end{equation}
\begin{lemma}\label{phi-inv-lem}
For all $h \in G_\infty$, we have 
$$\phi_\infty(h \lambda, h \nu) = \phi_\infty(\lambda, \nu).$$
\end{lemma}
\begin{proof}
The lemma follows from the observation that both $P$ and $\tilde\phi_\infty$ depend only on $q_Q(\lambda)$ and $B_Q(\lambda, \nu)$, and by definition $G_\infty$ preserves this.
\end{proof}
The group $\SL_2(\A)$ acts on $S(V_N(\A))$ via the Weil representation $\omega$ determined by the standard additive character $\psi$ on $\A/\Q$ such that $\psi_\infty(x) = {\rm exp}(2 \pi \sqrt{-1} x)$. If $\Phi \in S(V_N(\A))$, then the Weil representation is given by
\begin{equation}\label{Explicit-WeilAct}
\omega(\mat{a}{}{}{a^{-1}})\Phi(x) = |a|^{\frac N2+1}\Phi(ax), \quad \omega(\mat{1}{t}{}{1})\Phi(x) = \psi(tq_Q(x))\Phi(x), \quad \omega(\mat{}{1}{-1}{})\Phi(x) = \hat{\Phi}(x),
\end{equation}
where  $\hat{\Phi}$ denotes the Fourier transform of $\Phi$ with respect to a self-dual measure of $V_N(\A)$ with respect to $V_N(\A)^2\ni (x,y)\mapsto \psi(B_Q(x,y))\in\C^{(1)}$. In the first equation, we need to add a factor $\chi_{q_Q}(a):=\langle (-1)^{\frac N2 + 1}\det(Q),a\rangle$ on the right hand side. But, note that $\det(Q) = -1$ and $\frac N2+1$ is odd. Hence,  $\chi_{q_Q}(a) = 1$ for all $a$.

By $\omega_v$ we denote the $v$-component of $\omega$ at a place $v$. Let  $\hat{\Phi}$ denote the Fourier transform of $\Phi\in S(V_N(\Q_v))$ with respect to a self-dual measure of $V_N(\Q_v)$ with respect to $V_N(\Q_v)^2\ni(x,y)\mapsto \psi_v(B_Q(x,y))\in\C^{(1)}$. Over the Schwartz Bruhat space $S(V_N(\Q_v))$ at $v$ we also provide a description of $\omega_v$ as follows:
\begin{equation}\label{Explicit-WeilAct-local}
 \omega_v(\mat{a}{}{}{a^{-1}})\Phi(x) = |a|_v^{\frac N2+1}\Phi(ax), \quad \omega(\mat{1}{t}{}{1})\Phi(x) = \psi_v(tq_Q(x))\Phi(x), \quad \omega(\mat{}{1}{-1}{})\Phi(x) = \hat{\Phi}(x).
\end{equation}
In general, we have to add a factor $\gamma_{q_Q,v}^{-1}(1)$, which denotes the local constant called the Weil constant, to the right hand side of the last equation above. But in our case, since $L_p$ is self-dual for all $p<\infty$ and dimenstion of $V_N$ is $N+2$, we see that $\gamma_{q_Q,v}(1) = 1$ for all places $v$ (See \cite[Chapitre II]{We} for details).

Suppose $\tau = u+\sqrt{-1}v$ lies in the complex upper half plane, and let $g_\tau = \mat{1}{u}{}{1} \mat{v^{1/2}}{}{}{v^{-1/2}} \in \SL_2(\R)$. Then 
\begin{align}\label{omega-action}
\big(\omega(g_\tau) \tilde\phi_\infty\big)(\lambda, \nu) &= v^{\frac N4+\frac 12}{\rm exp}(2 \pi \sqrt{-1} u q_Q(\lambda)) \tilde\phi_\infty(\sqrt{v}\lambda, \nu) \nonumber\\
&= v^{\frac N4+\frac 12}{\rm exp}(2 \pi \sqrt{-1} u \big(q_Q(\lambda_{\nu^+})+q_Q(\lambda_{\nu^-})\big))   {\rm exp}\big(2 \pi v q_Q(\lambda_{\nu^-}) - 2 \pi v q_Q(\lambda_{\nu^+})\big) \nonumber\\
&= v^{\frac N4+\frac 12}{\rm exp}(2 \pi \sqrt{-1} \big(\tau q_Q(\lambda_{\nu^+}) + \bar\tau q_Q(\lambda_{\nu^-})\big)).
\end{align}

The $\SL_2(\A)$ action commutes with the action of $\cG(\A)$, which we denote by $\omega(h) \Phi(x) := \Phi(h^{-1}x)$. For $\nu \in \mathcal D, h \in \cG(\A_f)$ and $g \in \SL_2(\A)$, let $\theta(g, \nu, h)$ be the linear functional on $S(V_N(\A_f))$ defined by
\begin{equation}\label{theta-functional-defn}
S(V_N(\A_f)) \ni \phi \mapsto \theta(g, \nu, h; \phi) := \sum\limits_{\lambda \in V_N(\Q)} \omega(g)\Big(\phi_\infty(\cdot, \nu) \otimes \omega(h)\phi\Big)(\lambda).
\end{equation}


\begin{lemma}
Let $h_0 \in \cG(\Q)$ and $g_0 \in \SL_2(\Q)$. We have 
\begin{equation}\label{theta-left-inv}
\theta(g, h_0 \nu, h_0 h; \phi) = \theta(g, \nu, h; \phi), \qquad \theta(g_0g, \nu, h; \phi) = \theta(g, \nu, h; \phi).
\end{equation}
\end{lemma}
\begin{proof}
The statement for $h_0$ follows from Lemma \ref{phi-inv-lem}, and a change of variable. For the statement of $g_0$, we can look at the three cases. For $g_0 = \mat{1}{t}{}{1}$, the result follows from $\psi|_\Q \equiv 1$. For $g_0 = \mat{a}{}{}{a^{-1}}$, we get the result from a change of variable $\lambda \mapsto a \lambda$. For $g_0 = \mat{}{1}{-1}{}$, the result is obtained by Poisson summation.
\end{proof}
If $g_1 \in \SL_2(\A_f)$ and $h_1 \in \cG(\A_f)$, then we have
\begin{equation}\label{theta-right-inv}
\theta(gg_1, \nu, hh_1; \phi) = \theta(g, \nu, h; \omega(g_1, h_1)\phi),
\end{equation}
with $\omega(g_1,h_1)$ meaning $\omega(g_1)\omega(h_1)$. 
Hence, if $\phi \in S(V_N(\A_f))^{K_f}$, then the map 
$$(\nu, h) \mapsto \theta(g, \nu, h; \phi)$$
on $\mathcal D \times \cG(\A_f)$ descends to a function on $X = \Gamma_S \backslash \mathcal D^+$. We may view it as a linear functional on $S(V_N(\A_f))^{K_f}$, and obtain
$$\theta : \SL_2(\Q) \backslash \SL_2(\A) \times X \rightarrow \Big(S(V_N(\A_f))^{K_f}\Big)^\vee \qquad \qquad (g, \nu, h) \mapsto \theta(g, \nu, h; \cdot).$$
Since we have assumed that $L$ is even unimodular, it is self dual, and in this case $S(V_N(\A_f))^{K_f}$ is a one dimensional space spanned by the characteristic function $\phi_0$ of $\otimes_{p<\infty}L_{0,p}$.


Let $f \in  S(\SL_2(\Z);-\frac{r^2 + 1}{4})$ be as in Theorem \ref{LNP-theorem}. Write $g \in \SL_2(\A)$ as $g = \gamma g_\infty k$ with $\gamma \in \SL_2(\Q), g_\infty \in \SL_2(\R)$ and $k\in \prod_\ell \SL_2(\Z_\ell)$. Define $f_0 : \SL_2(\A) \to \C$ be $f_0(g) := f(g_\infty \langle \sqrt{-1} \rangle)$.

%

As a function of $g  \in \SL_2(\A)$, the function $g \to f_0(g) \bar\theta(g, \nu, h; \phi_0)$
is left-$\SL_2(\Q)$ invariant and right-$\SO(2, \R) \prod_\ell \SL_2(\Z_\ell)$ invariant using (\ref{theta-left-inv}), (\ref{theta-right-inv}) and $\hat{\phi_0} = \phi_0$ since $L$ is self-dual. This allows us to define the adelic Borcherds lift on $X$ by
\begin{equation}\label{adelic-Borcherds-lift}
\Phi(\nu, h, f_0) := \int\limits_{\SL_2(\Q) \backslash \SL_2(\A)} f_0(g) \overline{\theta}(g, \nu, h; \phi_0)  dg.
\end{equation}

\begin{proposition}\label{class-adelic-prop}
For $(x, y) \in H_N$, let $\nu(x,y) \in \mathcal D^+$ by (\ref{HN=D}).  Then we have
$$\Phi(\nu(x,y), 1, f_0) = F_f(n(x)a_y). 
$$
\end{proposition}
\begin{proof}
This follows from Section 3 of \cite{LNP}.
\end{proof}



\section{Petersson norm of the theta lifting}\label{PeterssonNormsec}
We study the Petersson norm 
\[
||\Phi(*,*,f_0)||^2:=\displaystyle\int_{\cG(\Q)\backslash \cG(\A)}\Phi(\nu,h,f_0)\overline{\Phi(\nu,h,f_0)}d\nu dh
\]
of $\Phi(\nu,h, f_0)$ by the well known approach originally due to S. Rallis~(cf.~\cite{Ra}). In view of Lemma \ref{Classnum-Cusps} and Proposition \ref{class-adelic-prop} 
this coincides with 
\[
||F_f||_2^2:=\displaystyle\int_{\Gamma_S\backslash G_{\infty}}F_f(h_{\infty})\overline{F_f(h_{\infty})}dh_{\infty},
\]
which is nothing but the Petersson norm of our lift in the classical setting. This is convergent since $\Phi(\nu, h, f_0)$ is cuspidal. \\
We have that  
\begin{align}
||\Phi(*,*,f_0)||^2 &=\displaystyle\int_{\cG(\Q)\backslash \cG(\A)}\displaystyle\int_{(\SL_2(\Q)\backslash \SL_2(\A))^2}f_0(g_1) \overline{\theta}(g_1, \nu, h; \phi_0)\overline{f_0(g_2) \overline{\theta}(g_2, \nu, h;\phi_0)}dg_1dg_2d\nu dh \nonumber\\
&= \displaystyle\int_{(\SL_2(\Q)\backslash \SL_2(\A))^2} f_0(g_1) \overline{f_0(g_2)} I(g_1, g_2, \phi_0) dg_1 dg_2, \label{Petersson-norm-eqn}
\end{align}
where
$$I(g_1, g_2, \phi_0) := \displaystyle\int_{\cG(\Q)\backslash \cG(\A)}  \overline{\theta}(g_1, \nu, h; \phi_0)  \theta(g_2, \nu, h; \phi_0) d\nu dh.$$

Since we are in the convergent range, the change in order of integration is justified. In view of the doubling variables of the Weil representation~(cf.~\cite[Section 11]{Hw}) we have
\begin{equation}\label{double-theta}
\bar{\theta}(g_1,\nu,h,\phi_{0})\theta(g_2,\nu,h,\phi_{0})=\theta((g_1, g_2),\delta(\nu),\delta(h),\phi_{0}\otimes\phi_{0}).
\end{equation}

Here, for $g_1,~g_2\in \SL_2$ we regard $(g_1,g_2)\in \SL_2\times \SL_2$ as its image of the diagonal embedding $\SL_2\times \SL_2\hookrightarrow \Sp_4$ given by 
$$\SL_2 \times \SL_2 \ni (\mat{a}{b}{c}{d}, \mat{a'}{b'}{c'}{d'}) \mapsto \begin{bmatrix}a&&-b&\\&a'&&b'\\-c&&d&\\&c'&&d'\end{bmatrix} \in \Sp_4,$$
and the map $\delta$ denotes the canonical diagonal embedding of $\cG$ into the orthogonal group defined by the quadratic space $(\Q^{N+2}\oplus\Q^{N+2}, q_Q\oplus q_Q)$, for which note that $\Sp_4\times\delta(\cG)$ forms a dual pair. 
 
On the right hand side of (\ref{double-theta}), the theta series is a linear functional on $S(V_N(\A)) \oplus S(V_N(\A))$ with the choice of Schwartz function being
  $\Phi_{0}:=(\phi_{\infty}\oplus\phi_{\infty})\otimes(\phi_{0}\oplus\phi_{0})$.

By the convergent Siegel-Weil formulas \cite[Th{\'e}or{\`e}me 5]{We2}, the integral $I(g_1, g_2, \phi_0)$ can be written as a special value of a certain Siegel Eisenstein series. Let us describe this next. Let $P(\A)=N(\A)M(\A)$ be the adelized Siegel parabolic subgroup of $\Sp_4(\A)$ with the Levi part $M(\A)\simeq \GL_{2}(\A)$ and the unipotent radical $N(\A)\simeq {\rm Sym}_{2}(\A)$, and let $\tilde{K}$ be the standard maximal compact sugroup $\prod_{p<\infty}\Sp_{4}(\Z_p)\times (\OO(4)(\R)\cap \Sp_{4}(\R))$.  The Weil representation $\omega^D$ of $\Sp_4(\A)$ on $S(V_N(\A)^2)$, obtained by the doubling of $\omega$, is realized as follows:
\begin{align}\label{Weil-Doubling}
\omega^D(
\mat
{A}{0_2}{0_2}{{}^tA^{-1}}
)\Phi(X)&=
|\det(A)|^{\frac N2+1}\Phi(XA)~(A\in \GL_2(\A)), \nonumber\\
\quad\omega^D(
\mat{1_2}{Y}{0_2}{1_2})\Phi(X)&=\psi(\frac{1}{2}\trace((X,X)\cdot Y))\Phi(X)~(Y\in{\rm Sym_2(\A)}),\nonumber\\
\omega^D(
\mat{0_2}{1_2}{-1_2}{0_2})\Phi(X)&=\hat{\Phi}(X),
\end{align} 
where $\Phi$ denotes a Schwarts Bruhat function on $V_N(\A) \oplus V_N(\A)$, and $(X, X):= (\frac 12 {}^tx_iQx_j) \in {\rm Sym}_2(\A)$ for $X = (x_1, x_2) \in V_N(\A)^2$.  
For the first formula we note that there is a factor $\chi_{q_Q}(\det(A))$, which is proved to be trivial.


We now also provide local representations $\omega^D_v$ of $\omega^D$ at each $v\le\infty$. With the same notation for $\omega_v$
\begin{align}\label{Weil-Doubling-local}
\omega^D_v(
\mat
{A}{0_2}{0_2}{{}^tA^{-1}}
)\Phi(X)&=
|\det(A)|^{\frac N2+1}\Phi(XA)~(A\in \GL_2(\Q_v)),\nonumber\\
\quad\omega^D_v(
\mat{1_2}{Y}{0_2}{1_2})\Phi(X)&=\psi(\frac{1}{2}\trace((X, X) \cdot Y))\Phi(X)~(Y\in{\rm Sym_2(\Q_v)}),\nonumber\\
\omega^D_v(
\mat{0_2}{1_2}{-1_2}{0_2})\Phi(X)&=
\hat{\Phi}(X),
\end{align} 
where $\Phi$ denotes a Schwarts Bruhat function on $V_N(\Q_v) \oplus V_N(\Q_v)$, and $(X, X):= (\frac 12 {}^tx_iQx_j) \in {\rm Sym}_2(\Q_v)$ for $X = (x_1, x_2) \in V_N(\Q_v)^2$. For the formula above we remark that there is the factor $\chi_{q_{Q,v}}(\det(A))$~(resp.\\
$\gamma_{q_Q,v}(1)^{-2}$) in the first formula~(resp.~third formula), which turn out to be trivial.

%

We have the Iwasawa decomposition $\Sp_{4}(\A)=P(\A) \tilde{K}$. Write any $g \in \Sp_4(\A)$ as $g = n m(a) k$ with some $n=
\begin{pmatrix}
1_2 & X\\
0_2 & 1_2
\end{pmatrix}~(X\in\Sym_2(\A)), m(a)=
\begin{pmatrix}
a & 0_2\\
0_2 & {}^ta^{-1}
\end{pmatrix}~(a\in \GL_2(\A))$ and $k\in \tilde{K}$. We then set $|a(g)| := |\det(a)|_\A$. This is a well-defined function on $\Sp_4(\A)$ that is left $N(\A)M(\Q)$-invariant and right $\tilde{K}$-invariant. For $s \in \C$ and $\Phi \in S(V_N(\A)^2)$, define $\Xi(g, s) := \omega^D(g)\Phi(0) \cdot |a(g)|^{s-s_0}$, where $s_0 = (N-1)/2$. From (\ref{Weil-Doubling}), we see that $\omega^D(nm(a)g)\Phi(0) = |\det(a)|^{\frac N2+1}\omega^D(g) \Phi(0)$. Recall that the modular character on the Siegel parabolic is $\delta_P(p) = |a(p)|^{3}$. Hence, we get 
$$\Xi(pg, s) = |a(p)|^s\delta_P(p)^{\frac 12} \Xi(g,s) \text{ for } p \in P(\A), g \in \Sp_4(\A).$$
Hence we can conclude that $\Xi \in {\rm Ind}_{P(\A)}^{\Sp_4(\A)}(\delta_P^{\frac s3})$, which is the normalized parabolic induction. For any $F \in {\rm Ind}_{P(\A)}^{\Sp_4(\A)}(\delta_P^{\frac s3})$, the Siegel Eisenstein series is defined by 
$$E(g, s; F) := \sum\limits_{\gamma \in P(\Q) \backslash \Sp_4(\Q)} F(\gamma g, s),$$
which converges absolutely for ${\rm Re}(s) > 3/2$~(cf.~\cite[Th{\'e}or{\`e}me 1]{We2}). 

We now state the Siegel Weil formula for our case as follows~(cf.~\cite[Th{\'e}or{\`e}me 5]{We2}):
\begin{proposition}
Let $\Phi_0 \in S(V_N(\A)) \oplus S(V_N(\A))$ be defined by $(\phi_\infty \oplus \phi_\infty) \otimes (\phi_0 \oplus \phi_0)$ as above. For $s \in \C, g \in \Sp_4(\A)$, let $\Xi_0(g, s) := \omega^D(g)\Phi_0(0) |a(g)|^{s-s_0}$, where $s_0 = (N-1)/2$. Then, for $g_1, g_2 \in \SL_2(\A)$, we have
\begin{equation}\label{SW-conv}
I(g_1, g_2, \phi_0) = E((g_1, g_2), s_0; \Xi_0).
\end{equation}
\end{proposition}

Returning to the Petersson norm calculation (\ref{Petersson-norm-eqn}), we now have
\begin{equation}\label{Pet-norm-eqn1}
||\Phi(*,*,f_0)||^2 = \int\limits_{\SL_2(\Q) \backslash \SL_2(\A)} \Big(\int\limits_{\SL_2(\Q) \backslash \SL_2(\A)} f_0(g_1) E((g_1, g_2), s_0; \Xi_0) dg_1 \Big) \overline{f_0(g_2)} dg_2.
\end{equation}
Call the inner integral $Z(s_0, f_0, \Xi_0)(g_2)$. By Theorem 3.6 of \cite{PSS20}, we have
$$Z(s_0, f_0, \Xi_0)(g_2) = \prod\limits_{v \leq \infty} Z_v(s_0, f_v, \Xi_v)(g_v),$$
where
\begin{equation}\label{local-zeta-int}
Z_v(s_0, f_v, \Xi_v)(g_v) = \int\limits_{\SL_2(\Q_v)} \Xi_v(Q_2(h,1), s_0) f_v(g_vh) dh
\end{equation}
with local components $\{f_v\}$ of $f_0$. 
Here 
$$Q_2 := \begin{bmatrix}1\\&&&1\\&&1&1\\1&-1\end{bmatrix}.$$ 
For $v < \infty$, all the data is unramified, and hence, by Proposition 4.1 of \cite{PSS20}, we get
\begin{equation}\label{unrami-comp}
Z_v(s_0, f_v, \Xi_v)(g_v) = \frac{L_v(s_0+\frac 12, \pi_v, {\rm Ad})}{\zeta_v(s_0+\frac 32)\zeta_v(2s_0+1)} f_v(g_v).
\end{equation}
Recall that $\pi_f = \otimes_v'\pi_v$ is the automorphic cuspidal representation of $\GL_2(\A)$ generated by $f_0$. Here, $L(s, \pi_v, {\rm Ad})$ is the degree $3$ adjoint $L$-function of $\pi_v$.

\subsection{The archimedean computation}\label{arch-comp-sec}
In this section we will compute the archimedean integral $Z_\infty(s_0, f_\infty, \Xi_\infty)$. We begin with the following proposition.
\begin{proposition}\label{arch-K-inv-prop}
Let $\tilde K_{\infty}$ be the maximal compact subgroup of $\Sp_4(\R)$ given by $\tilde K_{\infty}:=\Sp_4(\R)\cap \OO_4(\R)$, which is isomorphic to the unitary group $U(2)$ of degree two. 
For the Schwartz functions $\phi_{\infty}$ in (\ref{arch-phi}), set $\Phi_{\infty}:=\phi_{\infty}\oplus\phi_{\infty}$. Then $\Xi_\infty(g, s_0) := \omega^D(g) \Phi_\infty(0)$  is right $\tilde K_{\infty}$-invariant
(with respect to the Weil representation).
\end{proposition}
\begin{proof}
For the proof we need the following lemma.
\begin{lemma}\label{arch-lem1}
\begin{enumerate}
\item The Schwartz function $\phi_{\infty}$ in (\ref{arch-phi}) is right $\SO(2)$-invariant 
(with respect to the Weil representation).
\item Let $\tau_{\Lambda}$ be an irreducible representation of $\tilde K_{\infty}\simeq U(2)$ with dominant weight $\Lambda=(\lambda_1,\lambda_2)\in\Z^2$ with $\lambda_1\ge\lambda_2$. The infinitesimal action of $\tau_{\Lambda}$, also denoted by $\tau_{\Lambda}$, is described by the following explicit formula:
\begin{align*}
\tau_{\Lambda}(H_1)v_k&=(\lambda_2+k)v_k,~\tau_{\Lambda}(H_2)v_k=(\lambda_1-k)v_k,\\
\tau_{\Lambda}(X)v_k&=(k+1)v_{k+1},~\tau_{\Lambda}(\bar{X})v_k=(d-(k-1))v_{k-1},
\end{align*}
\end{enumerate}
where
\begin{itemize}
\item $\{v_k\}_{0\le k\le d:=\lambda_1-\lambda_2}$ denotes a set of weight vectors which forms a basis of the representation space for $\tau_{\Lambda}$.
\item \begin{align*}
&H_1:=
\begin{bmatrix}
0 & 0 & -\sqrt{-1} & 0\\
0 & 0 & 0 & 0\\
\sqrt{-1} & 0 & 0 & 0\\
0 & 0 & 0 & 0
\end{bmatrix},~H_2:=\begin{bmatrix}
0 & 0 & 0 & 0\\
0 & 0 & 0 & -\sqrt{-1}\\
0 & 0 & 0 & 0\\
0 & \sqrt{-1} & 0 & 0
\end{bmatrix}, \\
&X:=\frac{1}{2}
\begin{bmatrix}
0 & 1 & 0 &  -\sqrt{-1}\\
-1 & 0 & -\sqrt{-1} & 0\\
0 & \sqrt{-1} & 0 & 1\\
\sqrt{-1} & 0 & -1 & 0
\end{bmatrix},
\end{align*}
and $\bar{X}$ denotes the complex conjugate of $X$.
\end{itemize}
\end{lemma}
\begin{proof}
(i)~Putting $\nu_0=\nu(0,1)$ we see
\[
\phi_{\infty}(\lambda,\nu_0)=P(\lambda,\nu_0)\exp(-2\pi q_R(\lambda)),
\]
where $q_R$ denotes the quadratic form defined by the majorant $R:=
\begin{pmatrix}
1 & & \\
 & S & \\
 &    & 1
\end{pmatrix}$. 
It suffices to calculate the action of 
$\begin{bmatrix}
\cos\theta & \sin\theta\\
-\sin\theta & \cos\theta
\end{bmatrix}$ on $\phi_{\infty}(\lambda,\nu_0)$ in (\ref{arch-phi}) via the Weil representation.  In fact, we see that the calculation does not depend on the choice of $\nu\in \mathcal D$. This is due to the well known commutativity of the actions of the dual pair by the Weil representation, which is also ensured by the argement of Section \ref{adelic-Borcherdslift}. The main difficulty is the calculation of the action of 
$\begin{bmatrix}
0 & 1\\
-1 & 0
\end{bmatrix}$ on $\phi_{\infty}$. 
We remark that such a calculation is well known for the Gaussian multiplied by a harmonic polynopmial. For instance we can find it in the proof of the automorphy for holomorphic theta series with harmonic polynomials. 

When the matrix 
$\begin{bmatrix}
\cos\theta & \sin\theta\\
-\sin\theta & \cos\theta
\end{bmatrix}$ is diagonal i.e. $\sin\theta=0$ the calculation is settled in a straight forward manner. 
Let us thus assume $\sin\theta\not=0$. We then have that
\[
\begin{bmatrix}
\cos\theta & \sin\theta\\
-\sin\theta & \cos\theta
\end{bmatrix}=
\begin{bmatrix}
1 & -\frac{1}{\tan\theta}\\
0 & 1
\end{bmatrix}
\begin{bmatrix}
0 & 1\\
-1 & 0
\end{bmatrix}
\begin{bmatrix}
\sin\theta & 0\\
0 & \frac{1}{\sin\theta}
\end{bmatrix}
\begin{bmatrix}
1 & -\frac{1}{\tan\theta}\\
0 & 1
\end{bmatrix}.
\]
The action by $\begin{bmatrix}
\sin\theta & 0\\
0 & \frac{1}{\sin\theta}
\end{bmatrix}
\begin{bmatrix}
1 & -\frac{1}{\tan\theta}\\
0 & 1
\end{bmatrix}$ is calculated to be
$$
|\sin\theta|^{\frac N2+1}\exp(2\pi\sqrt{-1}(-\frac{1}{\tan\theta}q_Q(\sin\theta\cdot \lambda))(\exp(-\frac{\Delta}{8\pi})(-2^{\frac N2-4}(\sin\theta\cdot x_1)^{\frac N2}))\exp(-2\pi q_R(\sin\theta\cdot \lambda)),$$
where $x_1$ denotes the first entry of $\lambda\in \R^{N+2}$, more precisely the coordinate of $\R\cdot\nu_0=\{\frac{1}{\sqrt{2}}(x_1,0,x_1)\mid x_1\in\R\}$.  
This is deduced immediately from the definition of the Weil representation~(cf.~(\ref{Explicit-WeilAct-local})). 

As the next step we consider the action by $\begin{bmatrix}
0 & 1\\
-1 & 0
\end{bmatrix}$, which is given by the Fourier transform of the above expression multiplied by the Weil constant. For this purpose we introduce
$$\phi_{\tau,\infty}(x):=(\exp(-\frac{\Delta}{{8\pi\rm Im}(\tau)})(2^{\frac N2-3}x_1^{\frac N2}))\exp(2\pi\sqrt{-1}({\rm Re}(\tau)q_Q(\lambda)+{\rm Im}(\tau)\sqrt{-1}q_R(\lambda)))~(\tau\in{\frak h}).$$
The result of the action by $\begin{bmatrix}
\sin\theta & 0\\
0 & \frac{1}{\sin\theta}
\end{bmatrix}
\begin{bmatrix}
1 & -\frac{1}{\tan\theta}\\
0 & 1
\end{bmatrix}$ is rewritten as
$$|\sin\theta|^{N+1}\phi_{-\cos\theta\sin\theta+\sqrt{-1}\sin^2\theta,\infty}(x).$$
From \cite[Corollary 3.5]{Bo}  the Fourier transform of $\phi_{\tau, \infty}$ is as follows:
$$\widehat{\phi_{\tau,\infty}}(x)=
(\tau/\sqrt{-1})^{-\frac{1}{2}}(\sqrt{-1}\bar{\tau})^{-\frac{N+1}{2}}(-\tau)^{-\frac N2}\phi_{-1/\tau,\infty}(x).$$
Putting $\tau=-\cos\theta\sin\theta+\sqrt{-1}\sin^2\theta$ the result of the action by $\mat{}{1}{-1}{}$  is calculated to be 
\[
\phi_{-1/\tau,\infty}(x),
\]
taking into account the Weil constant for $q_Q$ and the constant factor of the self-dual measure, both of which are trivial for an even unimodular lattice. 
As the last step,  the action of 
$\begin{bmatrix}
1 & -\frac{1}{\tan\theta}\\
0 & 1
\end{bmatrix}$ is given by
\[
\phi_{-1/\tau-\frac{1}{\tan\theta},\infty}(x)=\phi_{\sqrt{-1},\infty}=\phi_{\infty}.
\]


(ii)~Let ${\frak k}$ and ${\frak u}(2)$ be the Lie algebra of $K$ and $U(2)$ respectively, and note the following isomorphism
\[
{\frak k}\ni
\begin{bmatrix}
A & B\\
-B & A
\end{bmatrix}\mapsto A+\sqrt{-1}B\in {\frak u}(2)
\]
as Lie algebras. This isomorphism is naturally extended to the isomorphism between their complexifications: ${\frak k}_{\C}\simeq{\frak u}(2)_{\C}$ , the latter of which coincides with the complex matrix algebra $M_2(\C)$ of degree two. 
Via this isomorphism, $H_1,~H_2,~X,$ and $\bar{X}$ correspond to $
\begin{bmatrix}
1 & 0\\
0 & 0
\end{bmatrix},~
\begin{bmatrix}
0 & 0\\
0 & 1
\end{bmatrix},~
\begin{bmatrix}
0 & 1\\
0 & 0
\end{bmatrix}$, and $
\begin{bmatrix}
0 & 0\\
1 & 0
\end{bmatrix}$ respectively. Recall that the irreducible representation of $U(2)$ with highest weight $\Lambda=(\lambda_1,\lambda_2)$ is realizeed by $\det^{\lambda_2}{\rm sym}^{\lambda_1-\lambda_2}(\R^2)$, where ${\rm sym}^{\lambda_1-\lambda_2}(\R^2)$ denotes the $(\lambda_1-\lambda_2)$-th symmetric tensor representation of the standard representation of $U(2)$. The formula is obtained by considering the pullback of the infinitesimal action of the irreducible representation of $U(2)$ with highest weight $\Lambda$.
\end{proof}
For the proof of Proposition \ref{arch-K-inv-prop}, we note that $\Xi_\infty(g, s_0)$ is right $\tilde K_\infty$-invariant if and only if $\Xi_\infty(g, s_0)$ belongs to the one-dimensional $\tilde K_\infty$-type $\tau_{\Lambda}$ with dominant weight $\Lambda = (0,0)$. By definition of the Lie algebra action, we have, for $i=1,2$
$$\tau_\Lambda(\sqrt{-1}H_i)\Xi_\infty(g, s_0) := \frac d{d\theta}\Big|_{\theta=0} \Xi_\infty(g \exp(\theta\sqrt{-1}H_i), s_0).$$
Since 
$$\exp(\theta\sqrt{-1}H_1) = \begin{bmatrix}\cos(\theta)&0&\sin(\theta)&0\\0&1&0&0\\-\sin(\theta)&0&\cos(\theta)&0\\0&0&0&1\end{bmatrix}, \qquad \exp(\theta\sqrt{-1}H_2) = \begin{bmatrix}1&0&0&0\\0&\cos(\theta)&0&\sin(\theta)\\0&0&1&0\\0&-\sin(\theta)&0&\cos(\theta)\end{bmatrix},$$
part (i) of  Lemma \ref{arch-lem1} implies that $\tau_\Lambda(\sqrt{-1}H_i)\Xi_\infty(g, s_0) = 0$ for $i=1,2$. Hence, by part (ii) of  Lemma \ref{arch-lem1}, we see that $\Xi_\infty(g, s_0)$ is the weight $(0,0)$ vector in $\tau_{(\lambda, -\lambda)}$ for $\lambda \geq 0$. In the terminology of  part (ii) of  Lemma \ref{arch-lem1}, $\Xi_\infty(g, s_0)$ is the vector $v_{\lambda}$ in the basis of $\tau_{(\lambda, -\lambda)}$.

For $\theta \in \R$, set $r(\theta) := \mat{\cos(\theta)}{\sin(\theta)}{-\sin(\theta)}{\cos(\theta)}$ and $R(\theta) = \mat{r(\theta)}{}{}{{}^tr(\theta)^{-1}}$. By definition of the local Weil representation (\ref{Weil-Doubling-local}), we see that $(\omega^D(R(\theta))\Phi_\infty)(Y) = \Phi_\infty(Yr(\theta))$. Note that, if $Y = (y_1, y_2) \in V_N(\R)^2$, then $Yr(\theta) = (\cos(\theta)y_1-\sin(\theta)y_2, \sin(\theta)y_1+\cos(\theta)y_2)$. Using the definition of $\phi_\infty$ and the fact that $$q_Q\big(\cos(\theta)y_1-\sin(\theta)y_2\big) + q_Q\big(\sin(\theta)y_1+\cos(\theta)y_2\big) = q_Q(y_1) + q_Q(y_2),$$
we see that $\Phi_\infty(Yr(\theta)) = \Phi_\infty(Y)$. Hence we see that $\Xi_\infty(gR(\theta), s_0) = \Xi_\infty(g, s_0)$ for all $g$ and $\theta$. Since $\exp(\sqrt{-1}\theta(X + \bar{X})) = R(\theta)$, this implies that $\tau_\Lambda(\sqrt{-1}(X + \bar{X})) \Xi_\infty = 0$. By part (ii) of Lemma \ref{arch-lem1}, we can conclude that $\lambda = 0$, as required. This completes the proof of the proposition. 
\end{proof}

We will now compute the archimedean integral (\ref{local-zeta-int})
$$Z_\infty(s_0, f_\infty, \Xi_\infty)(g_2) = \int\limits_{\SL_2(\R)} \Xi_\infty(Q_2(h, 1), s_0) f_\infty(g_2h) dh.$$

 Note that $\pi_\infty$ is the irreducible principle series of $\GL_2^+(\R)$ with all even $\SO(2)$-types, and the Maass cusp form $f$ is viewed as the weight $0$ vector $f_{\infty}$ in $\pi_\infty$. Hence, in the induced model for $\pi_\infty$, we have
\begin{equation}\label{phi-formula}
f_\infty(\mat{a}{\ast}{}{a^{-1}} k) = a^{\sqrt{-1}r + 1}, \quad a \in \R^+, k \in \SO(2,\R).
\end{equation}
Here $r$ depends on the Laplace eigenvalue of the Maass form $f$. By Proposition \ref{arch-K-inv-prop}, the local section $\Xi_\infty$ is right $\tilde K_{\infty}$-invariant.  The local integral $Z_\infty$ is an element of $\pi_\infty$ and a simple change of variable shows that it is also right $\SO(2)$-invariant. Since, $f_\infty$ is the unique (up to scalars) element of $\pi_\infty$ which is right $\SO(2)$-invariant, we see that
$$Z_\infty(s_0, f_\infty, \Xi_\infty)(g_2) = B(s_0) f_\infty(g_2).$$
The final task is to compute $B(s_0)$ which can be achieved by 
$$B(s_0) = Z_\infty(s_0, f_\infty, \Xi_\infty)(1) = \int\limits_{\SL_2(\R)} \Xi_\infty(Q_2(h,1), s_0) f_\infty(h) dh.$$
Using the Iwasawa decomposition, we see that, if $u$ is a right $\SO(2)$-invariant function on $\SL_2(\R)$, then
$$\int\limits_{\SL_2(\R)} u(g) dg = 2 \pi\int\limits_0^\infty \int\limits_{-\infty}^\infty u(\mat{a}{}{}{a^{-1}} \mat{1}{x}{}{1}) a^{-1}dx da.$$
Hence, we get
\begin{align*}
B(s_0) &= 2 \pi\int\limits_0^\infty \int\limits_{-\infty}^\infty \Xi_\infty(Q_2(\mat{a}{}{}{a^{-1}} \mat{1}{x}{}{1},1), s_0) f_\infty(\mat{a}{}{}{a^{-1}} \mat{1}{x}{}{1})a^{-1}dx da\\
&= 2 \pi\int\limits_0^\infty \int\limits_{-\infty}^\infty \Xi_\infty(Q_2(\mat{a}{}{}{a^{-1}} \mat{1}{x}{}{1},1), s_0) a^{\sqrt{-1}r} dx da.
\end{align*}
We need to first simplify the integrand. For this, note that
$$Q_2 = \begin{bmatrix}1\\&1\\&1&1\\1&&&1\end{bmatrix} \begin{bmatrix}1\\&&&1\\&&1\\&-1\end{bmatrix}.$$
Hence, we have
\begin{align*}
&\Xi_\infty(Q_2(\mat{a}{}{}{a^{-1}} \mat{1}{x}{}{1},1), s_0) \\
=& \Xi_\infty(\begin{bmatrix}1\\&1\\&1&1\\1&&&1\end{bmatrix} \begin{bmatrix}1\\&&&1\\&&1\\&-1\end{bmatrix} \begin{bmatrix}a&&-ax\\&1\\&&a^{-1}\\&&&1\end{bmatrix}, s_0) \\
=& \Xi_\infty(\begin{bmatrix}a\\&1\\&&a^{-1}\\&&&1\end{bmatrix} \begin{bmatrix}1\\&1\\&a&1\\a&&&1\end{bmatrix} \begin{bmatrix}1\\&&&1\\&&1\\&-1\end{bmatrix} \begin{bmatrix}1&&-x\\&1\\&&1\\&&&1\end{bmatrix}, s_0) \\
=& a^{s_0+3/2} \Xi_\infty(\begin{bmatrix}1\\&1\\&a&1\\a&&&1\end{bmatrix} \begin{bmatrix}1\\&&&1\\&&1\\&-1\end{bmatrix} \begin{bmatrix}1&&-x\\&1\\&&1\\&&&1\end{bmatrix}, s_0) \\
=& a^{s_0+3/2} \Xi_\infty(\begin{bmatrix}1\\&1\\&a&1\\a&&&1\end{bmatrix}  \begin{bmatrix}1&&-x\\&1\\&&1\\&&&1\end{bmatrix} \begin{bmatrix}1\\&&&1\\&&1\\&-1\end{bmatrix}, s_0) \\
=& a^{s_0+3/2} \Xi_\infty(\begin{bmatrix}1\\&1\\&a&1\\a&&&1\end{bmatrix}  \begin{bmatrix}1&&-x\\&1\\&&1\\&&&1\end{bmatrix}, s_0).
\end{align*}
Here, we have used that $\Xi_\infty$ lies in the induced representation and is right $\tilde{K}_{\infty}$-invariant. Now, write
$$\mat{1}{}{a}{1} = \mat{1}{x'}{}{1} \mat{y^{1/2}}{}{}{y^{-1/2}} r(\theta) \text{ with } x' = \frac a{1+a^2}, y = \frac 1{1+a^2}, e^{\sqrt{-1}\theta} = \frac{1-\sqrt{-1}a}{\sqrt{1+a^2}}, r(\theta) = \mat{c(\theta)}{s(\theta)}{-s(\theta)}{c(\theta)}.$$
Here, $c(\theta) = \cos(\theta)$ and $s(\theta) = \sin(\theta)$. Hence,
\begin{align*}
&\Xi_\infty(Q_2(\mat{a}{}{}{a^{-1}} \mat{1}{x}{}{1},1), s_0) \\
=& a^{s_0+3/2} \Xi_\infty(\begin{bmatrix}1&&&x'\\&1&x'\\&&1\\&&&1\end{bmatrix} \begin{bmatrix}y^{1/2}\\&y^{1/2}\\&&y^{-1/2}\\&&&y^{-1/2}\end{bmatrix} \begin{bmatrix}c(\theta)&&&s(\theta)\\&c(\theta)&s(\theta)\\&-s(\theta)&c(\theta)\\-s(\theta)&&&c(\theta)\end{bmatrix} \\ 
& \qquad \qquad   \qquad \qquad \qquad \qquad\begin{bmatrix}1&&-x\\&1\\&&1\\&&&1\end{bmatrix}, s_0) \\
=& \frac{a^{s_0+3/2}}{(1+a^2)^{s_0+3/2}} \Xi_\infty(\begin{bmatrix}c(\theta)&&&s(\theta)\\&c(\theta)&s(\theta)\\&-s(\theta)&c(\theta)\\-s(\theta)&&&c(\theta)\end{bmatrix} \begin{bmatrix}1&&-x\\&1\\&&1\\&&&1\end{bmatrix}, s) \\
=&\big(\frac a{1+a^2}\big)^{s_0+3/2} \Xi_\infty(\begin{bmatrix}1&-s(\theta)c(\theta)x&-c(\theta)^2x\\&1\\&&1\\&s(\theta)^2x&s(\theta)c(\theta)x&1\end{bmatrix} \begin{bmatrix}c(\theta)&&&s(\theta)\\&c(\theta)&s(\theta)\\&-s(\theta)&c(\theta)\\-s(\theta)&&&c(\theta)\end{bmatrix}, s)\\
=&\big(\frac a{1+a^2}\big)^{s_0+3/2} \Xi_\infty(\begin{bmatrix}1&-s(\theta)c(\theta)x&-c(\theta)^2x\\&1\\&&1\\&s(\theta)^2x&s(\theta)c(\theta)x&1\end{bmatrix}, s)\\
=&\big(\frac a{1+a^2}\big)^{s_0+3/2} \Xi_\infty(\begin{bmatrix}1&-s(\theta)c(\theta)x&-c(\theta)^2x\\&1\\&&1\\&&s(\theta)c(\theta)x&1\end{bmatrix} \begin{bmatrix}1\\&1\\&&1\\&s(\theta)^2x&&1\end{bmatrix}, s)\\
=&\big(\frac a{1+a^2}\big)^{s_0+3/2} \Xi_\infty(\begin{bmatrix}1\\&1\\&&1\\&s(\theta)^2x&&1\end{bmatrix}, s).
\end{align*}
Note that $s(\theta)^2 = a^2/(1+a^2)$. Substituting this in the integral for $B(s_0)$, we get
$$B(s_0) = 2 \pi\int\limits_0^\infty \int\limits_{-\infty}^\infty \big(\frac a{1+a^2}\big)^{s_0+3/2} \Xi_\infty(\begin{bmatrix}1\\&1\\&&1\\&s(\theta)^2x&&1\end{bmatrix}, s_0) a^{\sqrt{-1}r} dx da. $$
A change of variable $x \to s(\theta)^{-2}x$ gives us
\begin{align*}
B(s_0) &= 2 \pi\int\limits_0^\infty \int\limits_{-\infty}^\infty \big(\frac a{1+a^2}\big)^{s_0+3/2} \frac{1+a^2}{a^2}\Xi_\infty(\begin{bmatrix}1\\&1\\&&1\\&x&&1\end{bmatrix}, s_0) a^{\sqrt{-1}r} dx da \\
&= 2 \pi\int\limits_0^\infty \frac{a^{s_0-1/2+\sqrt{-1}r}}{(1+a^2)^{s_0+1/2}} da \int\limits_{-\infty}^\infty \Xi_\infty(\begin{bmatrix}1\\&1\\&&1\\&x&&1\end{bmatrix}, s_0) dx.
\end{align*}
Let us consider the integral in the $a$ variable first. Change of variable $u = a^2$ gives us
$$\int\limits_0^\infty \frac{a^{s_0-1/2+\sqrt{-1}r}}{(1+a^2)^{s_0+1/2}} da = \frac 12 \int\limits_0^\infty \frac{u^{\frac{s_0}2-\frac 34+\frac{\sqrt{-1}r}2}}{(1+u)^{s_0+1/2}} du.$$
Recall the beta function 
$$B(v, w) = \frac{\Gamma(v)\Gamma(w)}{\Gamma(v+w)}.$$
One integral representation for the beta function (see pg 7 of \cite{MOS}) is 
$$B(v,w) = \int\limits_0^\infty \frac{u^{v-1}}{(1+u)^{v+w}} du \text{ for } {\rm Re}(v) > 0, {\rm Re}(w) > 0.$$
We have $v = s_0/2+1/4+\sqrt{-1}r/2$ and $w = s_0/2+1/4-\sqrt{-1}r/2$. Hence, we get
$$\int\limits_0^\infty \frac{a^{3s-1/2+\sqrt{-1}r}}{(1+a^2)^{3s+1/2}} da = \frac 12  \frac{\Gamma(s_0/2+1/4+\sqrt{-1}r/2)\Gamma(s_0/2+1/4-\sqrt{-1}r/2)}{\Gamma(s_0+1/2)}.$$
Now, let us compute the integral in the $x$ variable. We will again use the Iwasawa decomposition 
$$\mat{1}{}{x}{1} = \mat{1}{x'}{}{1} \mat{y^{1/2}}{}{}{y^{-1/2}} r(\theta) \text{ with } x' = \frac x{1+x^2}, y = \frac 1{1+x^2}, e^{\sqrt{-1}\theta} = \frac{1-\sqrt{-1}x}{\sqrt{1+x^2}}.$$
Hence
\begin{align*}
\int\limits_{-\infty}^\infty \Xi_\infty(\begin{bmatrix}1\\&1\\&&1\\&x&&1\end{bmatrix}, s_0) dx &= \int\limits_{-\infty}^\infty \Xi_\infty(\begin{bmatrix}1\\&1&&x'\\&&1\\&&&1\end{bmatrix} \begin{bmatrix}1\\&y^{1/2}\\&&1\\&&&y^{-1/2}\end{bmatrix}, s_0) dx \\
&= \int\limits_{-\infty}^\infty \frac 1{(1+x^2)^{\frac{s_0}2 + \frac 34}} dx = 2  \int\limits_0^\infty \frac 1{(1+x^2)^{\frac{s_0}2 + \frac 34}} dx.
\end{align*}
The change of variable $x = \tan(\theta)$ gives us 
$$\int\limits_{-\infty}^\infty \Xi_\infty(\begin{bmatrix}1\\&1\\&&1\\&x&&1\end{bmatrix}, s_0) dx = 2 \int\limits_0^{\pi/2} \frac 1{(\sec^2(\theta))^{\frac{s_0}2 + \frac 34}} \sec^2(\theta) d\theta = 2 \int\limits_0^{\pi/2} \cos(\theta)^{s_0-\frac 12} d\theta$$
On pg 8 of \cite{MOS}, we have the formula
$$\int\limits_0^{\pi/2} \sin(\theta)^{2x-1} \cos(\theta)^{2y-1} d\theta = \frac 12 B(x, y) \text{ for } {\rm Re}(x) > 0, {\rm Re}(y) > 0.$$
We have $x = 1/2$ and $y = s_0/2+1/4$. 
Hence, we have 
$$\int\limits_{-\infty}^\infty \Xi_\infty(\begin{bmatrix}1\\&1\\&&1\\&x&&1\end{bmatrix}, s_0) dx = B(1/2, s_0/2+1/4) = \frac{\Gamma(1/2)\Gamma(s_0/2+1/4)}{\Gamma(s_0/2+3/4)}.$$
Putting all this together, and using properties of the gamma function, we get
\begin{align*}
B(s_0) &= \pi \frac{\Gamma(s_0/2+1/4+\sqrt{-1}r/2)\Gamma(s_0/2+1/4-\sqrt{-1}r/2)}{\Gamma(s_0+1/2)} \frac{\Gamma(1/2)\Gamma(s_0/2+1/4)}{\Gamma(s_0/2+3/4)} \\
&= 2^{1-\frac{N}{2}} \pi^2 \frac{\Gamma(\frac N4+\frac{\sqrt{-1}r}2) \Gamma(\frac N4-\frac{\sqrt{-1}r}2)}{\Gamma(\frac N4 + \frac 12)^2}.
\end{align*}

Putting together (\ref{Pet-norm-eqn1}), (\ref{unrami-comp}) and the formula of $B(s_0)$ above, we get the following theorem.

\begin{theorem}\label{main-norm-thm}
The Petersson norm of the Borcherds theta lift is given by
\begin{equation}\label{lift-norm}
||\Phi(*,*,f_0)||^2 = \frac{L(\frac N2, \pi_f, {\rm Ad})}{\zeta(\frac N2 + 1)\zeta(N)} \Big(2^{1-\frac{N}{2}}\pi^2 \frac{\Gamma(\frac N4+\frac{\sqrt{-1}r}2) \Gamma(\frac N4-\frac{\sqrt{-1}r}2)}{\Gamma(\frac N4 + \frac 12)^2}\Big) ||f_0||^2.
\end{equation}
\end{theorem}

The above theorem together with Proposition \ref{class-adelic-prop} gives us Theorem \ref{norm-thm-intro} from the introduction. We can now obtain the important corollary regarding injectivity of the Borcherds theta lift.

\begin{corollary}\label{inj-cor}
Let $f \in S(\SL_2(\Z);-\frac{r^2 + 1}{4})$ and $F_f$ be its Borcherds lift defined in Section \ref{Autom-form}. Then the map $f \to F_f$ is injective.
\end{corollary}
\begin{proof}
If $f$ is a Hecke eigenform, then by Theorem \ref{LNP-theorem}, so is $F_f$. By Proposition \ref{class-adelic-prop}, the Petersson norm of $F_f$ is given by Theorem \ref{main-norm-thm} above. Hence, $f \to F_f$ is an injective map when restricted to Hecke eigenforms. To conclude the same about non-Hecke eigenforms, we follow the exact argument as in the proof of Theorem 7.1 of \cite{NPW}. Note that, for this we need formulas for Hecke eigenvalues of $F_f$ in terms of those of $f$, and these are obtained in Theorem 4.11 of \cite{LNP}.
\end{proof}

\section{Sup-norm bounds for Maass cusp forms on $O(1,N+1)$}\label{Supnormsec}
In this section, will obtain lower and upper bounds for the sup-norm of $F_f$. For the lower bounds we will use the Fourier expansion of $F_f$. For the upper bound, we will use a combination of Fourier expansion of $F_f$ and the pre trace formula method.

\subsection{Upper and lower bounds using Fourier expansions}\label{Four-exp-sec}
Recall from Section \ref{Autom-form} that we have $f \in S(\SL_2(\Z);-\frac{r^2 + 1}{4})$ given by the Fourier expansion
\begin{align*}
f(u+i v) &= \sum_{n\not=0}c(n)W_{0,\frac{\sqrt{-1}r}{2}}(4\pi|n|v)\exp(2\pi\sqrt{-1}nu) \\
&= \sum_{n\not=0}2 c(n) |n|^{1/2} v^{1/2} K_{\frac{\sqrt{-1}r}{2}}(2\pi|n|v)\exp(2\pi\sqrt{-1}nu).
\end{align*}
We have the Borcherds theta lift $F_f$ obtained in Section \ref{Autom-form} given by the Fourier expansion
\begin{equation}\label{Four-exp-2}
F_f(n(x)a_y)=\sum_{\lambda\in L\setminus\{0\}}A(\lambda)y^{\frac N2}K_{\sqrt{-1}r}(4\pi|\lambda|_Sy)\exp(2\pi\sqrt{-1}{}^t\lambda Sx)
\end{equation}
with 
$$A(\lambda)=|\lambda|_S\sum_{d|d_{\lambda}}c\left(-\frac{|\lambda|_S^2}{d^2}\right)d^{\frac N2-2}.$$

Let us make the assumption that $f$ is a Hecke eigenform with $c(m) = \pm c(-m)$  for all $m \in \Z$. Let us also assume that $r$ is bounded from below, i.e. $r \gg 1$.

\subsection*{Fourier expansion of $F_f$ at the cusps of $\Gamma_S$}
In addition to the Fourier expansion (\ref{Four-exp-2}) of $F_f$ at the cusp at $\infty$ (a neighborhood of $y=\infty$ in $H_N$), we will also need the same at other cusps of $\Gamma_S$. Let us briefly recall this from Section 3.3 of \cite{LNP}. Let the adelization of $F_f$ (denoted by the same symbol) be as defined in equations (3.2)-(3.5) of \cite{LNP}. We know from Lemma \ref{Classnum-Cusps} that there is a bijection between the cusps of $\Gamma_S$ and the double cosets $\cH(\Q) \backslash \cH(\A) / U_fU_\infty$, which has representatives in $\cH(\A_f)$. Given $c \in \cG(\Q)$, a representative of a cusp of $\Gamma_S$, let us write $c = c_fc_\infty$ with $c_f \in \cG(\A_f)$ and $c_\infty \in G_\infty$. Let $h \in \cH(\A_f)$ correspond to the cusp $c$. Then, there exist $\gamma \in \cH(\Q)$ and $k \in U_\infty$ such that $c = \begin{bmatrix}1\\&\gamma h k\\&&1\end{bmatrix}$. We will write $\gamma_f$ for the finite part of $\gamma$. Using the fact that $F_f$ is left invariant  by $\cG(\Q)$, right invariant by $K_\infty$ and $\begin{bmatrix}1\\&h_\infty\\&&1\end{bmatrix} \in K_\infty$ for $h_\infty \in H_\infty$, we get for $x \in \R^N, y \in \R_{>0}$,
\begin{align*}
F_f(n(x)a_y\begin{bmatrix}1\\&h\\&&1\end{bmatrix}) &=  F_f(n(x)a_y\begin{bmatrix}1\\&\gamma_f^{-1}\\&&1\end{bmatrix} c_f) = F_f(n(x)a_y\begin{bmatrix}1\\&\gamma^{-1}\\&&1\end{bmatrix} c_f) \\
&= F_f(\begin{bmatrix}1\\&\gamma^{-1}\\&&1\end{bmatrix} n(\gamma x)a_y c_f) = F_f(c_fn(\gamma x)a_y) \\
&= F_f(cc_\infty^{-1} n(\gamma x)a_y) = F_f(c_\infty^{-1} n(\gamma x)a_y),
\end{align*}
where we regard $\gamma$ as $\gamma_{\infty}$ for $n(\gamma x)$ appearing in the equations of the second and third lines. 
The definition of the adelic $F_f(n(x)a_y\begin{bmatrix}1\\&h\\&&1\end{bmatrix})$ given in (3.2), (3.3) of \cite{LNP} implies
\begin{equation}\label{cusp-Fourier-coeff}
F_f(c_\infty^{-1} n(x)a_y) = \sum\limits_{\lambda \in L_h \backslash \{0\}} A_h(\lambda) y^{\frac N2} K_{\ii r}(4 \pi |\lambda|_Sy) \exp(2 \pi \ii {}^t\lambda S (\gamma^{-1}x)),
\end{equation}
where $L_h = aL$, with $h = au, (a,u) \in \GL_N(\Q) \times \big(\prod_{p<\infty} \SL_N(\Z_p) \times \SL_N(\R)\big)$, and 
$$A_h(\lambda) = |\lambda|_S\sum_{d|d_{\lambda}}c\left(-\frac{|\lambda|_S^2}{d^2}\right)d^{\frac N2-2},$$
with $d_\lambda$ being the smallest positive integer such that $\frac 1{d_\lambda} a^{-1}\lambda \in L$. Note that $L_h$ is also an even unimodular lattice and (\ref{cusp-Fourier-coeff}) gives the Fourier expansion of $F_f$ at the cusp corresponding to $c^{-1}$.

\subsection*{Bounds on $||F_f||_2, |A_h(\lambda)|$ and $K_{\ii r}(y)$}
 Note that $c(1) \neq 0$ and we have the following estimate on the first Fourier coefficient $c(1)$ given by \cite[Theorem 2]{Iw} and \cite[Corollary 0.3]{Ho-Lo}. For any $\epsilon > 0$
\begin{equation}\label{c(1)-estimates}
r^{-\epsilon} \cosh(\pi r/2) \ll_\epsilon  \frac{|c(1)|^2}{||f||_2^2} \ll_\epsilon r^\epsilon \cosh(\pi r/2).
\end{equation}
Here $||f||_2$ is the Petersson norm of $f$, and we are not assuming it to be equal to $1$. 
\begin{proposition}\label{Petersson-norm-estimate}
For any $\epsilon>0$ we have an estimate as follows:
\begin{enumerate}
\item $$\sqrt{\sinh(\pi r/2)}r^{-\frac N4+\frac 12} \ll_N \frac{||f||_2}{||F_f||_2}\ll_{N} \sqrt{\sinh(\pi r/2)} r^{-\frac N4+\frac 12}.$$
\item For all $\lambda \in L_h \backslash \{0\}$ and $\epsilon > 0$, we have 
\begin{align*}
\frac{|A_h(\lambda)|}{||f||_2}  &\ll_\epsilon |\lambda|^{2\theta+1+\epsilon}_S d_\lambda^{\frac N2-2-2\theta} r^{\epsilon}\sqrt{\cosh(\pi r/2)} \\
& \ll_\epsilon |\lambda|^{\frac N2-1+\epsilon}_S  r^{\epsilon}\sqrt{\cosh(\pi r/2)},
\end{align*}
where $\theta = 7/64$ is the current best estimate towards the Ramanujan conjecture for Maass forms.  
\end{enumerate}
\end{proposition}
\begin{proof}
By Theorem \ref{main-norm-thm}, we have
$$||F_f||_2^2 = \frac{L_(\frac N2, \pi_f, {\rm Ad})}{\zeta(\frac N2 + 1)\zeta(N)} \Big(2^{1-\frac{N}{2}}\pi^2 \frac{\Gamma(\frac N4+\frac{\sqrt{-1}r}2) \Gamma(\frac N4-\frac{\sqrt{-1}r}2)}{\Gamma(\frac N4 + \frac 12)^2}\Big) ||f||_2^2.$$
We have
$$1 \ll_N  \frac{2^{1-\frac{N}{2}}\pi^2L_(\frac N2, \pi_f, {\rm Ad})}{\zeta(\frac N2 + 1)\zeta(N)\Gamma(\frac N4 + \frac 12)^2}  \ll_N 1.$$
Since $N/2 > 1 + 2\theta$, the $L$-function above is given by a convergent Dirichlet series, and hence can be bounded by estimates independent of $f$.

%
%
%
%
%
%
For the terms involving the Gamma function we use the following standard properties 
$$
\Gamma(\frac N4+\frac{\sqrt{-1}r}{2})\Gamma(\frac N4-\frac{\sqrt{-1}r}{2}) = |\Gamma(\frac N4+\frac{\sqrt{-1}r}{2})|^2 = \frac{\pi r/2}{\sinh(\pi r/2)} \prod\limits_{k=1}^{\frac N4-1}(k^2 + (r/2)^2). $$
Hence, we have
$$ \frac{\pi r/2}{\sinh(\pi r/2)} r^{\frac N2-2} \ll_N \Gamma(\frac N4+\frac{\sqrt{-1}r}{2})\Gamma(\frac N4-\frac{\sqrt{-1}r}{2}) \leq \frac{\pi r/2}{\sinh(\pi r/2)} (N^2+r^2)^{\frac N4-1} \ll_N \frac{\pi r/2}{\sinh(\pi r/2)} r^{\frac N2-2}.
$$
%
Here, we have used $r \gg 1$. Hence we obtain the estimate for $||f||_2/||F_f||_2$ in the statement of the proposition.

To obtain the estimate for $|A_h(\lambda)|$, we have
$$|A_h(\lambda)| \leq |\lambda|_S \sum_{d|d_{\lambda}}|c\left(-\frac{|\lambda|_S^2}{d^2}\right)|d^{\frac N2-2} = |\lambda|_S \sum_{d|d_{\lambda}}|c(1)||\mu\left(\frac{|\lambda|_S^2}{d^2}\right)|d^{\frac N2-2},$$
where $\mu(m)$ is the Hecke eigenvalue for $f$ for the Hecke operator $T(m)$. We have, by \cite[Appendix 2, Proposition 2]{Ki}~(see also \cite[Section 8.5]{Iw2}),
$|\mu(m)| \leq m^\theta \tau(m)$ where $\tau(m)$ is the number of divisors of $m$ and $\theta = 7/64$ is the current best estimate towards the Ramanujan conjecture for Maass forms. 
Using $\tau(m) \ll_\epsilon m^{\epsilon}$ for every $m > 0$, the right hand inequality of (\ref{c(1)-estimates}) and $\frac N2-2-2\theta-\epsilon > 0$, we have
\begin{align*}
|A_h(\lambda)| &\leq |\lambda|_S \sum_{d|d_{\lambda}}|c(1)| \tau\left(\frac{|\lambda|_S^2}{d^2}\right) \left(\frac{|\lambda|_S^2}{d^2}\right)^\theta d^{\frac N2-2} \\
& \ll_\epsilon   |\lambda|^{2\theta+1+\epsilon}_S |c(1)| \sum_{d|d_{\lambda}} d^{\frac N2-2-2\theta-\epsilon} \\
& \ll_\epsilon |\lambda|^{2\theta+1+\epsilon}_S |c(1)| d_\lambda^{\frac N2-2-2\theta} \\
& \ll_\epsilon |\lambda|^{2\theta+1+\epsilon}_S d_\lambda^{\frac N2-2-2\theta} ||f||_2r^{\epsilon}\sqrt{\cosh(\pi r/2)} \\
& \ll_\epsilon |\lambda|^{\frac N2-1+\epsilon}_S  ||f||_2 r^{\epsilon}\sqrt{\cosh(\pi r/2)},
\end{align*}
which completes the proof of the proposition.\\
\end{proof}

We will require the estimate for the $K$-Bessel function. 

\begin{lemma}
\label{Besselbd}

Let $r \gg 1$ be a real number.

\begin{enumerate}

\item If $1 \ll y < r$, we have
\[
e^{\pi r /2}K_{\ii r}(y) \ll r^{-1/4} (r - y)^{-1/4}.
\]

\item If $y = r + O(r^{1/3})$ and $y \gg 1$, we have
\[
e^{\pi r /2}K_{\ii r}(y) = \pi (2/y)^{1/3} {\rm Ai}(\xi e^{-2\pi i /3}) + O(y^{-2/3}) \ll r^{-1/3},
\]
where $\xi = i(y-r) (-iy/2)^{-1/3}$.

\item If $2r > y > r$, we have
\[
e^{\pi r /2}K_{\ii r}(y) \ll r^{-1/4} (y - r)^{-1/4} \exp( -C r^{-1/2} ( y-r)^{3/2})
\]
for some $C > 0$.

\item If $y \ge 2r$, we have
\[
e^{\pi r /2}K_{\ii r}(y) \ll \exp( -C y)
\]
for some $C > 0$.

\end{enumerate}

\end{lemma}

\begin{proof}

The bound (i) follows from the asymptotic of Erd\'{e}lyi \cite[7.13.2, (19)]{Erdelyi} (after noting that $r+y \sim r$ in the range under consideration), and formula (ii) is due to Balogh \cite[(8)]{Bal}.  We shall derive (iii) and (iv) from the formula \cite[7.13.2, (18)]{Erdelyi} of Erd\'{e}lyi, which gives
\[
e^{\pi r /2}K_{\ii r}(y) \ll (y^2-r^2)^{-1/4} \exp( \pi r /2 -(y^2 - r^2)^{1/2} - r \sin^{-1}(r/y) ).
\]
To derive (iii) from this, we must show that the argument of the exponential satisfies
\begin{equation}
\label{besselexpt1}
- \pi r/2 + (y^2 - r^2)^{1/2} + r \sin^{-1}(r/y) \gg r^{-1/2} (y-r)^{3/2}
\end{equation}
when $2r > y > r$.  To do this, define $\rho = y/r$.  If we define $f(\rho) = -\pi/2 + (\rho^2-1)^{1/2} + \sin^{-1}(\rho^{-1})$, then we have
\[
- \pi r/2 + (y^2 - r^2)^{1/2} + r \sin^{-1}(r/y) = r f(\rho).
\]
We have $f'(\rho) = \sqrt{\rho^2-1}/\rho \sim \sqrt{\rho - 1}$ when $2 > \rho > 1$, and $f(1) = 0$.  This implies that $f(\rho) \sim (\rho-1)^{3/2}$ in this range, which gives (\ref{besselexpt1}).

To establish (iv), we note that when $\rho \ge 2$ we have $f'(\rho) = \sqrt{1 - \rho^{-2} } \ge 1/2$, which implies that $f(\rho) \ge f(2) + (\rho-2)/2$ in this range.  Because $f(2) > 0$, this implies that $f(\rho) \gg \rho$, and hence that
\[
- \pi r/2 + (y^2 - r^2)^{1/2} + r \sin^{-1}(r/y) = rf(\rho) \gg y.
\]
Hence, we get 
$$e^{\pi r /2}K_{ir}(y) \ll (y^2-r^2)^{-1/4} e^{-Cy} \ll y^{-1/2} e^{-Cy} \ll e^{-Cy},$$
since $y \geq 2r$ and $r \gg 1$.
\end{proof}

%
%

Finally, we need a lemma to estimate the number of vectors in the lattice $L_h$ with prescribed norms.
\begin{lemma}\label{lambda-estimate-lem}
Let $L_h$ be the even unimodular lattices as above with dimension $N$. Set $k = N/2-1$. Then, for $m \in \Z_{>0}, k' \in \R_{>0}$ and $\epsilon > 0$, we have 
\begin{equation}\label{lambda-esimate-eqn}
\sum_{\substack{\lambda \in L_h \\ |\lambda|^2_S = m}} d_\lambda^{k-k'} \ll_{L, \epsilon} m^{k+\epsilon},
\end{equation}
where we note that the implied constant is independent of $h$.
\end{lemma}
\begin{proof}
We have
\begin{align*}
\sum_{\substack{\lambda \in L_h \\ |\lambda|_S^2 = m}} d_\lambda^{k-k'} &= \sum_{d^2|m} \sum_{\substack{|\lambda|_S^2 = m \\ d_\lambda = d}} d^{k-k'} = \sum_{d^2|m} d^{k-k'} \sum_{\substack{|\lambda|_S^2 = m \\ d_\lambda = d}} 1 \\
&= \sum_{d^2|m} d^{k-k'} \sum_{\substack{|\lambda|_S^2 = m/d^2 \\ d_\lambda = 1}} 1 \leq \sum_{d^2|m} d^{k-k'} \sum_{|\lambda|_S^2 = m/d^2} 1.
\end{align*}
Since $L_h$ is assumed to be an even unimodular lattice of dimension $N$, ~\cite[p.109, Corollary 2]{Se} and \cite[Proposition 1.3.5]{Bu} implies that, for any $M > 0$, we have $\sharp\{\lambda\in L_h \mid |\lambda|_S^2=M\} \ll_L \tau(M) M^{\frac N2-1}$. Note that we have $\tau(M) \ll_\epsilon  M^\epsilon$. Using this we get
\begin{align*}
\sum_{\substack{\lambda \in L_h \\ |\lambda|_S^2 = m}} d_\lambda^{k-k'} &\ll_{L, \epsilon} \sum_{d^2|m} d^{k-k'} \Big( \frac m{d^2} \Big)^{k+\epsilon} = \sum_{d^2|m} \frac{m^{k+\epsilon}}{d^{k+k'-\epsilon}} \\
&\ll_{L, \epsilon} \sum_{d^2|m} m^{k+\epsilon} \leq m^{k+\epsilon} \tau(m) \ll_{L,\epsilon} m^{k+\epsilon},
\end{align*}
as required.
\end{proof}



\subsection*{Upper bound for sup-norm using Fourier expansion of $F_f$}
In this section we will obtain an upper bound for $|F_f(c_\infty^{-1}n(x)a_y)|/||F_f||_2$ using the Fourier expansion (\ref{cusp-Fourier-coeff}) of $F_f$ and the bounds on $||F_f||_2, |A_h(\lambda)|, K_{\sqrt{-1}r}(y)$ and bounds on lattice points obtained in Proposition \ref{Petersson-norm-estimate} and Lemmas \ref{Besselbd} and \ref{lambda-estimate-lem}.
\begin{theorem}\label{Upper-bd-thm}
Let $f \in  S(\SL_2(\Z);-\frac{r^2 + 1}{4})$ be a non-zero Hecke eigenform with Fourier coefficients $c(-m) = \pm c(m)$ for all $m \in \Z$. Let $F_f \in \cM(\Gamma_S, \sqrt{-1}r)$ be the Borcherds theta lift of $f$.  Assume that $r$ is bounded below, i.e., $r \gg 1$. 
For any $\epsilon > 0, x \in \R^N$, and any cusp $c$ of $\Gamma_S$, the following holds:
\[
\frac{1}{ \| F_f \|_2 } | F_f( c_\infty^{-1} n(x) a_y) | \ll_{\epsilon, N, L} 
\begin{cases}
y^{-N/2 -1 - 2\theta} r^{3N/4 + 1 +2\theta + \epsilon} & 1 \ll y \le r^{11/12}; \\
y^{-N/2 +1 - 2\theta} r^{3N/4 - 5/6 + 2\theta + \epsilon} & r^{11/12} < y \le r/2\pi; \\
e^{-Cy} & r / 2\pi < y. 
\end{cases}
\]
\end{theorem}

\begin{proof}
Using the Fourier expansion of $F_f$ given by (\ref{cusp-Fourier-coeff}), we see that it suffices to bound
\[
\frac{1}{ \| F_f \|_2 } \sum_{\lambda \in L_h} |A_h(\lambda)| y^{N/2} K_{\ii r}(4 \pi | \lambda|_S y).
\]
Using Proposition \ref{Petersson-norm-estimate} and Lemma \ref{lambda-estimate-lem}, this is bounded by
\begin{equation}
\label{Besselsum}
S_y := r^{-N/4 + 1/2 + \epsilon} y^{N/2} \sum_{m \ge 1} m^{N/2 - 1/2 + \theta + \epsilon} e^{\pi r /2} K_{\ii r}(4 \pi \sqrt{m} y).
\end{equation}
We break the sum in (\ref{Besselsum}) into four ranges and denote $S_y := S_y^{(1)} + S_y^{(2)} + S_y^{(3)} + S_y^{(4)}$, where $S_y^{(1)}$ is obtained by summing over  $4 \pi \sqrt{m} y \le r/2$, i.e. $m \le (r / 8 \pi y)^2$, $S_y^{(2)}$ is obtained by summing over $r/2 < 4 \pi \sqrt{m} y \le r$, i.e. $(r / 8 \pi y)^2 < m \le (r / 4 \pi y)^2$, $S_y^{(3)}$ is obtained by summing over $r < 4 \pi \sqrt{m} y \le 2r$ and $S_y^{(4)}$ is obtained by summing over $2r < 4 \pi \sqrt{m} y$.

{\bf Computing $S_y^{(1)}$}: In this range, the bound (i) from Lemma \ref{Besselbd} becomes $e^{\pi r /2} K_{\ii r}(4 \pi \sqrt{m} y) \ll r^{-1/2}$, so we have 
\begin{align*}
S_y^{(1)} &\ll r^{-N/4 + 1/2 + \epsilon} y^{N/2} r^{-1/2} \sum_{1 \le m \le (r / 8 \pi y)^2} m^{N/2 - 1/2 + \theta + \epsilon} \\
&\ll r^{-N/4 + \epsilon} y^{N/2} (r/y)^{N + 1 + 2\theta + \epsilon} = \frac{r^{\frac{3N}4 + 1 + 2\theta + \epsilon}}{y^{\frac N2 + 1 + 2\theta}}.
\end{align*}

{\bf Computing $S_y^{(2)}$}:  In this range we have $m \sim (r/y)^2$, so that
\begin{equation}
\label{Besselsum2}
S_y^{(2)} \ll 
r^{-N/4 + 1/2 + \epsilon} y^{N/2} (r/y)^{N - 1 + 2\theta + \epsilon} \sum_{(r / 8 \pi y)^2 < m \le (r / 4 \pi y)^2} e^{\pi r /2} K_{\ii r}(4 \pi \sqrt{m} y).
\end{equation}

If the sum on the right hand side of (\ref{Besselsum2}) is nonempty, it's highest term is $m_* = \lfloor (r / 4 \pi y)^2 \rfloor$.  Lemma \ref{Besselbd} implies that $e^{\pi r /2} K_{\ii r}(4 \pi \sqrt{m_*} y) \ll r^{-1/3}$ for all $y$, and so applying this to bound the contribution of $m_*$ to the sum gives
\[
\sum_{(r / 8 \pi y)^2 < m \le (r / 4 \pi y)^2} e^{\pi r /2} K_{\ii r}(4 \pi \sqrt{m} y) \ll r^{-1/3} + \sum_{(r / 8 \pi y)^2 < m \le m_* - 1} e^{\pi r /2} K_{\ii r}(4 \pi \sqrt{m} y).
\]
For the remaining terms in the sum, we apply Lemma \ref{Besselbd} (i), which gives
\[
\sum_{(r / 8 \pi y)^2 < m \le m_* - 1} e^{\pi r /2} K_{\ii r}(4 \pi \sqrt{m} y) \ll r^{-1/4} \sum_{(r / 8 \pi y)^2 < m \le m_* - 1} (r - 4 \pi \sqrt{m} y)^{-1/4}.
\]
Because the function $x \to (r - 4 \pi \sqrt{x} y)^{-1/4}$ is increasing on the interval $((r / 8 \pi y)^2, m_*)$, we may bound the sum by an integral as follows:
\begin{align*}
r^{-1/4} \sum_{(r / 8 \pi y)^2 < m \le m_* - 1} (r - 4 \pi \sqrt{m} y)^{-1/4} & < r^{-1/4} \int_{(r / 8 \pi y)^2}^{m_*} (r - 4 \pi \sqrt{x} y)^{-1/4} dx \\
& \le r^{-1/4} \int_{(r / 8 \pi y)^2}^{(r / 4 \pi y)^2} (r - 4 \pi \sqrt{x} y)^{-1/4} dx.
\end{align*}
We then have
\begin{align*}
r^{-1/4} \int_{(r / 8 \pi y)^2}^{(r / 4 \pi y)^2} (r - 4 \pi \sqrt{x} y)^{-1/4} dx & = r^{-1/4} \int_{ r / 8 \pi y }^{r / 4 \pi y} (r - 4 \pi u y)^{-1/4} 2u du \\
& \le r^{3/4} /2\pi y \int_{ r / 8 \pi y }^{r / 4 \pi y} (r - 4 \pi u y)^{-1/4} du \\
& = r^{3/4} /8 \pi^2 y^2 \int_{r/2}^r (r - u)^{-1/4} du \\
& \ll r^{3/2} y^{-2}.
\end{align*}
Combining these gives
\[
S_y^{(2)} \ll \frac{r^{\frac{3N}4 - \frac 12 + 2\theta + \epsilon}}{y^{\frac N2 - 1 + 2\theta}} ( r^{-1/3} + r^{3/2} y^{-2}).
\]
We note that this term dominates the contribution from $S_y^{(1)}$.

{\bf Computing $S_y^{(3)}$}:  Following the computation of  $S_y^{(2)}$, we can see that we get the same bound for $S_y^{(3)}$ as we got for $S_y^{(2)}$ above. To see this, note that we again have $m \sim (r/y)^2$, so it suffices to bound the sum
\[
\sum_{(r / 4 \pi y)^2 < m \le (r / 2 \pi y)^2} e^{\pi r /2} K_{\ii r}(4 \pi \sqrt{m} y).
\]
As in the case of $S_y^{(2)}$, we may bound the extremal term in this sum by $r^{-1/3}$.  For the other terms, we may use the bound $e^{\pi r /2} K_{\ii r}(4 \pi \sqrt{m} y) \ll r^{-1/4} (y-r)^{-1/4}$ coming from Lemma \ref{Besselbd} (iii), which gives a contribution of $\ll r^{3/2} y^{-2}$ as in the case of $S_y^{(2)}$.

{\bf Computing $S_y^{(4)}$}:  Here, we use part (iv) of Lemma \ref{Besselbd}, which gives
\[
S_y^{(4)} \ll 
r^{-N/4 + 1/2 + \epsilon} y^{N/2} \sum_{(r / 2 \pi y)^2 < m} m^{N/2 - 1/2 + \theta + \epsilon} \exp(-C \sqrt{m} y).
\]
Because $y \gg 1$ we have $\sqrt{m} y \gg \sqrt{m} + y$, so
\begin{align*}
S_y^{(4)} & \ll 
r^{-N/4 + 1/2 + \epsilon} y^{N/2} e^{-Cy} \sum_{(r / 2 \pi y)^2 < m} m^{N/2 - 1/2 + \theta + \epsilon} \exp(-C \sqrt{m}) \\
& \ll r^{-N/4 + 1/2 + \epsilon} y^{N/2} e^{-Cy} \sum_{m \ge 1} m^{N/2 - 1/2 + \theta + \epsilon} \exp(-C \sqrt{m}) \\
& \ll r^{-N/4 + 1/2 + \epsilon} e^{-Cy} \ll e^{-Cy},
\end{align*}
where the final inequality follows because $y \gg r$.

Finally, we combine these estimates to obtain the proposition.  First note that, when $y > r/2\pi$, the sums $S_y^{(1)}, S_y^{(2)}$ and $S_y^{(3)}$ are empty, and so we get
$$\frac{|F_f(c_\infty^{-1}n(x)a_y)|}{||F_f||_2} \ll S_y = S_y^{(4)} \ll e^{-Cy},$$
as required.   When $y \leq r/2\pi$, we can check that the contribution from $S_y^{(2)}$ dominates the one from $S_y^{(4)}$, and hence we get 
$$\frac{|F_f(c_\infty^{-1}n(x)a_y)|}{||F_f||_2} \ll S_y \ll \frac{r^{\frac{3N}4 - \frac 12 + 2\theta + \epsilon}}{y^{\frac N2 - 1 + 2\theta}} ( r^{-1/3} + r^{3/2} y^{-2}).$$
%
%
The term $y^{-2} r^{3/2}$ dominates when $y \le r^{11/12}$, while $r^{-1/3}$ dominates when $y > r^{11/12}$, which gives the first two bounds of the theorem.
\end{proof}

\subsection*{Lower bound for the sup-norm}

We now use the Fourier expansion to obtain a lower bound for the sup norm of $F_f$.  When doing this, it will be convenient to work in a cusp $c$ such that the corresponding unimodular lattice $L_h$ has a vector of length one.  As all even unimodular lattices form a single genus, we may do this by taking $L_h$ to be a sum of copies of the $E_8$ lattice.

Recall again that we have the Fourier expansion
$$F_f(c_\infty^{-1} n(x)a_y) = \sum\limits_{\lambda \in L_h \backslash \{0\}} A_h(\lambda) y^{\frac N2} K_{\ii r}(4 \pi |\lambda|_Sy) \exp(2 \pi \ii {}^t\lambda S (\gamma^{-1}x)).$$
Fix $\lambda_0 \in L_h$ to be an element of norm 1.  We then have $A_h(\lambda_0) = c(-1) = \pm c(1) \neq 0$. Using the above Fourier expansion of $F_f$, the Cauchy-Schwartz inequality, and the fact that $L$ is a unimodular lattice, we get
\begin{align*}
||F_f||_\infty &= {\rm vol}(\R^N/L)^{-1} \int\limits_{\R^N/L} ||F_f||_\infty dx \\
&\geq  \int\limits_{\R^N/L} |F_f(c_\infty^{-1} n(x)a_y)| |\exp(-2\pi\sqrt{-1}{}^t\lambda_0 S (\gamma^{-1}x))| dx \\
&\geq \Big|\int\limits_{\R^N/L} F_f(c_\infty^{-1} n(x)a_y) \exp(-2\pi\sqrt{-1}{}^t\lambda_0 S (\gamma^{-1}x)) dx \Big| \\
&= |A_h(\lambda_0) y^{N/2} K_{\sqrt{-1}r}(4\pi y)|.
\end{align*}
By part (ii) of Lemma \ref{Besselbd}, there is a value of $y$ such that $4\pi y  = r + O(r^{1/3})$, and $K_{\sqrt{-1}r}(4\pi y) \gg e^{-\pi r /2} r^{-1/3}$.  Choosing this value of $y$ and using $|A_h(\lambda_0)| =  |c(1)|$ gives
\begin{equation}\label{lower-bound-formula}
||F_f||_\infty \gg  |c(1)| r^{N/2-1/3} e^{-\pi r /2}.
\end{equation}
The lower bound of Theorem \ref{supnorm-thm-intro} now follows from (\ref{lower-bound-formula}) using the bound (\ref{c(1)-estimates}), and part i) of Proposition \ref{Petersson-norm-estimate}.


\subsection{Upper bounds using the pre-trace formula}\label{Tr-formula-sec}

In this section, we use a pre-trace inequality to obtain upper bounds on the lifted form $F_f$. The bound we prove holds for any square-integrable Laplace eigenfunction on any hyperbolic orbifold $X$ of finite volume with the Laplacian $\Delta$. To state it, we will need the notion of the height of a point $x \in X$ in the cusp, denoted by ${\rm ht}(x)$, which we recall in this section.  We note that only for this section, $\ii r$ will denote the spectral parameter of an eigenfunction on $X$, instead of the parameter of the Maass form $f$.  We note that these two spectral parameters are the same size, so this should not lead to any confusion.

\begin{theorem}
\label{trace-thm}

Let $X$ be a finite-volume hyperbolic orbifold of dimension $N+1$.  Let $\psi \in L^2(X)$ be an $L^2$-normalized Laplace eigenfunction with spectral parameter $\ii r$, so that $(\Delta + r^2 - N^2/4)\psi = 0$.  We have
\[
\psi(x) \ll (1 + |r|)^{N/2} + {\rm ht}(x)^{N/2} (1 + |r|)^{N/4}, \quad x \in X.
\]

\end{theorem}

We shall assume in the proof that $r > 1$, as the other cases are similar.

\subsubsection*{Background on hyperbolic geometry}
\label{hyperbolic-sec}

We let $G^0$ be the group of isometries of $H_N$, which can be identified with the index two subgroup of ${\rm O}(1,N)$ preserving the upper sheet of the two-sheeted hyperboloid.  Let $d$ be the standard distance function on $H_N$, and let $\partial H_N \simeq S^N$ denote the boundary sphere.

Let $X = \Gamma \backslash H_N$, where $\Gamma < G^0$ is a lattice.  Let $O(\Gamma) \subset \partial H_N$ denote the set of fixed points of parabolic elements of $\Gamma$.  We fix a set of representatives $\Xi$ for the $\Gamma$-orbits in $O(\Gamma)$, which can be identified with the set of cusps of $X$.  We let $\Gamma_\xi$ be the stabilizer of $\xi \in \Xi$ in $\Gamma$.  For each $\xi \in \Xi$ we choose a horoball $B_\xi$ tangent to the boundary at $\xi$, and denote $V_\xi = \Gamma_\xi \backslash B_\xi$.  By \cite[Thm 12.7.4]{Rat}, we may choose each $B_\xi$ so that the following hold:

\begin{enumerate}

\item $V_\xi$ is mapped isometrically to its image in $X$.

\item We have $X = X_0 \coprod V_\xi$, where $X_0$ is compact.

\item If $\gamma \in \Gamma  - \Gamma_\xi$, then $B_\xi \cap \gamma B_\xi = \emptyset$.

\end{enumerate}
\noindent
We note that \cite{Rat} states this theorem for manifolds, rather than orbifolds, but one may easily obtain the result for an orbifold by using Selberg's lemma to pass to a finite-index torsion-free subgroup of $\Gamma$.  We now define the height function ${\rm ht} : X \to \R_{\ge 1}$.  If $x \in X_0$, we set ${\rm ht}(x) = 1$.  Next, suppose that $x \in V_\xi$ for some $\xi$.  Let $B_\infty$ denote the standard horoball $\{ (x,y) : x \in \R^N, y \ge 1 \}$, and choose $g_\xi \in G^0$ such that $g_\xi B_\xi = B_\infty$.  We define ${\rm ht}(x) = y(g_\xi x)$~(the $y$-coordinate of $g_\xi x$), which is independent of the choice of $g_\xi$.  This has the key property that for any $C > 0$, the set of points $x \in X$ with ${\rm ht}(x) \le C$ is compact.

\subsubsection*{Test functions}

In this section, we construct a test function for use in the pre-trace inequality.  We shall do this using the Harish-Chandra transform, which we now recall.  Define $\varphi_s$ to be the standard spherical function on $H_N$ or $G^0$ with spectral parameter $s$.  We continue to normalize $s$ so that $\sqrt{-1}\R$ is the tempered axis.  Let $K \simeq {\rm O}(N)$ be the standard maximal compact subgroup of $G^0$.  For a $K$-biinvariant function $k \in C^\infty_c(G^0)$, we define its Harish-Chandra transform by
\[
\widehat{k}(s) = \int_{G^0} k(g) \varphi_s(g) dg.
\]
This is inverted by
\[
k(g) = C_N \int_{\sqrt{-1}\R} \widehat{k}(s) \varphi_s(g) |c(s)|^{-2} ds
\]
for a constant $C_N$, where $c(s)$ is Harish-Chandra's $c$-function.  We may now define the test function we shall use.

\begin{lemma}
\label{testfn}

There exists a $K$-biinvariant function $k_r \in C^\infty_c(G^0)$ with the following properties:

\begin{enumerate}[label=(\roman*)]

\item\label{P1} $k_r$ is supported in a fixed compact set that is independent of $r$.

\item\label{P2} $\widehat{k}_r(s) \ge 0$ for $s \in \sqrt{-1}\R \cup (0,N/2]$.

\item\label{P3} $\widehat{k}_r(\ii r) = 1$.

\item\label{P4} $k_r( g) \ll r^{N}(1 + r d(g,e))^{-N/2}$.  In particular, $\| k_r \|_\infty \ll r^{N}$.

\end{enumerate}

\end{lemma}

\begin{proof}

Let $h \in C^\infty(\C)$ be a function of Paley--Wiener type (i.e. the Fourier transform of a function in $C^\infty_c(\R)$) that is real and non-negative on $\R$, and satisfies $h(0) = 1$. Define $h_r$ by $h_r(s) = h(r - s) + h(r + s)$. 
$h_r$ is real valued on $\R$, and satisfies $h_r(r) \ge 1$.  We also have $h_r(\overline{s}) = \overline{h}_r(s)$ and $h_r(s) = h_r(-s)$ so $h_r$ is real-valued on $\ii \R$.

Define $k_r$ to be the $K$-biinvariant function on $G^0$ satisfying $\widehat{k}_r(\ii s) = h_r^2(s)$.  Because $h_r$ is of Paley-Wiener type, \ref{P1} follows from the Paley-Wiener theorem of Gangolli--Varadarajan~(cf.~\cite[Theorem 6.6.8]{GV}). For $s \in \ii \R \cup (0,N/2]$, we have that $h_r(s)$ is real, and hence  $\widehat{k}_r(s) = h_r^2(s) \ge 0$.  We also have $\widehat{k}_r(\ii r) = h_r^2(r) \ge 1$, and so we may arrange that $\widehat{k}_r(\ii r) = 1$ by scaling $h_r$.

Finally, \ref{P4} follows from inverting the Harish-Chandra transform and applying bounds for the spherical function.  We have
\[
k_r(g) = \int_{\ii \R} \widehat{k}_r(s) \varphi_s(g) |c(s)|^{-2} ds.
\]
We have $|c(s)|^{-2} \ll (1 + |s|)^{N}$ for $s \in \ii \R$.  By \cite[Theorem 1.3]{Mar}, or by applying \cite{Duistermaat} together with stationary phase, we have
\[
\varphi_{s}(g) \ll (1 + |s| d(g,e))^{-N/2}
\]
for all $s \in \ii \R$ and $g$ in the support of $k_r$, which implies
\[
k_r(g) \ll \int_{\ii \R} \widehat{k}_r(s) (1 + |s| d(g,e))^{-N/2} (1 + |s|)^{N} ds.
\]
The rapid decay of $\widehat{k}_r(s)$ away from $s = \pm \ii r$ effectively truncates the integral to the region where $s \sim \pm \ii r$, which gives \ref{P4}.

\end{proof}

\subsubsection*{The pre-trace inequality}

Let $k_r$ be as in Lemma \ref{testfn}.  The fundamental inequality we shall use to bound $\psi$ is
\begin{equation}
\label{pretrace}
| \psi(x) |^2 \le \sum_{\gamma \in \Gamma} k_r(x^{-1} \gamma x).
\end{equation}
This may be derived from the pre-trace formula, by using the positivity property \ref{P2} of $\hat{k}_r$ to drop all terms on the spectral side other than $| \psi(x) |^2$ (including the continuous spectrum).  For this we remark that the parameters of the discrete spectrum are contained in $\sqrt{-1}\R\cup(0,N/2]$, which parametrize equivalence classes of irreducible unitary spherical principal series representations of $O(1,N)$ together with spherical complimentary series representations. To confirm this fact on the representation theory of $O(1,N)$ we refer to \cite{Hi} and \cite[pp31--32, Remarks]{HT} for instance.

It may also be proved in an elementary way by an application of Cauchy--Schwartz and unfolding, see for instance \cite[Lemma 6.5]{BrumleyMar}, with the test function $\omega$ there taken to be the function $k_r^0$ satisfying $\widehat{k}_r^0 = h_r$ so that $k_r^0 * (k_r^0)^* = k_r$.

Let $R > 0$ be a constant, independent of $r$, such that the support of $k_r$ is contained in the ball of radius $R$ about the origin $(0,1)$ in $H_N$.  For $x$ in any fixed compact subset of $X$, (\ref{pretrace}) and \ref{P4} give $|\psi(x)| \ll r^{N/2}$, so we may assume that ${\rm ht}(x) > e^R > 1$.  In particular, $x$ is contained in a cusp neighbourhood $V_\xi$, which we shall assume to be fixed for the rest of the proof.  Moreover, until further notice we identify $x$ with a choice of lift $x \in B_\xi$.  Under this assumption, we shall show that only $\gamma \in \Gamma_\xi$ contribute to (\ref{pretrace}).  Indeed, suppose that $\gamma \in \Gamma$ satisfies $k_r(x^{-1} \gamma x) \neq 0$.  This implies that $d(\gamma x, x) \le R$.  This, together with ${\rm ht}(x) > e^R$, implies that $\gamma x \in B_\xi$, so that $\gamma B_\xi \cap B_\xi \neq \emptyset$, and hence that $\gamma \in \Gamma_\xi$ as required.

We now assume that our cusp $\xi$ is the standard point at infinity, and denote $\Gamma_\xi$ and $B_\xi$ by $\Gamma_\infty$ and $B_\infty$, as the proof for the other cusps is similar. We therefore have
\[
| \psi(x) |^2 \le \sum_{\gamma \in \Gamma_\infty} k_r(x^{-1} \gamma x).
\]
We next apply our upper bound \ref{P4} for $k_r$, which gives
\[
| \psi(x) |^2 \ll r^{N} \sum_{ \substack{ \gamma \in \Gamma_\infty \\ d(\gamma x,x) < R } } (1 + r d(\gamma x,x))^{-N/2}.
\]
We may identify $\Gamma_\infty$ with a lattice in ${\rm Isom}(\R^N)$.  By \cite[Thm 7.5.2]{Rat}, $\Gamma_\infty$ has a finite index subgroup of translations of rank $N$, which we identify with a lattice $L \subset \R^N$.  The action of $L$ on $H_N$ will be written additively.  We let $\gamma_1, \ldots, \gamma_k$ be coset representatives for $L \backslash \Gamma_\infty$, and define $x_i = \gamma_i x$.  We therefore have
\[
| \psi(x) |^2 \ll r^{N} \sum_{i = 1}^k \sum_{ \substack{ \ell \in L \\ d(\ell + x_i,x) < R } } (1 + r d(\ell + x_i,x))^{-N/2}.
\]
Theorem \ref{trace-thm} now follows from the following lemma.

\begin{lemma}

Let $v_1, v_2 \in \R^N$, and let $x_i = (v_i, y) \in H_N$.  We have
\[
\sum_{ \substack{ \ell \in L \\ d(\ell + x_1,x_2) < R } } (1 + r d(\ell + x_1,x_2))^{-N/2} \ll 1 + y^N r^{-N/2}.
\]

\end{lemma}

\begin{proof}

We may assume without loss of generality that $v_1, v_2$ lie in a fixed fundamental domain for $\R^N / L$.  It follows from an elementary computation that if $d(\ell + x_1,x_2) < R$ then $\| \ell \| \ll_R y$, and
\[
d(\ell + x_1,x_2) \sim_R \| \ell + v_1 - v_2 \|/y =  \| \ell \| / y + O(1/y).
\]
In particular, there is $C > 0$ such that for $\| \ell \| > C$ we have $d(\ell + x_1,x_2) \sim \| \ell \| / y$.  The finitely many $\ell$ with $\| \ell \| \le C$ contribute $O(1)$ to the sum, so it suffices to show that
\begin{equation}
\label{ellsum}
\sum_{ \substack{ \ell \in L \\ C < \| \ell \| \ll y } } (1 + r d(\ell + x_1,x_2))^{-N/2} \sim \sum_{ \substack{ \ell \in L \\ C < \| \ell \| \ll y } } (1 + r \| \ell \|/y )^{-N/2} \ll 1 + y^N r^{-N/2}.
\end{equation}
We break the sum
\[
\sum_{ \substack{ \ell \in L \\ C < \| \ell \| \ll y } } (1 + r \| \ell \|/y )^{-N/2}
\]
into those $\ell$ with $\| \ell \| < y/r$, and the compliment.  If $\| \ell \| < y/r$, we have $1 + r \| \ell \|/y \sim 1$, so the contribution these terms make is asymptotically bounded by $\# \{ \ell \in L : \| \ell \| < y/r \} \ll 1 + y^N r^{-N}$ which is smaller than the right hand side of (\ref{ellsum}).  If $\| \ell \| \ge y/r$, we have $1 + r \| \ell \|/y \sim r \| \ell \|/y$, so the contribution from these terms is bounded by
\[
\sum_{ \substack{ \ell \in L \\ C < \| \ell \| \ll y } } (r \| \ell \|/y )^{-N/2} \ll y^{N/2} r^{-N/2} \sum_{ \substack{ \ell \in L \\ C < \| \ell \| \ll y } } \| \ell \|^{-N/2} \ll y^N r^{-N/2}
\]
which is again less than the right hand side of (\ref{ellsum}).  This completes the proof.

\end{proof}

\subsection{Proof of Theorem \ref{supnorm-thm-intro}}\label{thm-pf-sec}
In this section we will give the proof of Theorem \ref{supnorm-thm-intro} from the introduction. First note that Theorem \ref{trace-thm} gives 
$$|F(c_\infty^{-1} n_x a_y)| / \| F_f \|_2 \ll r^{N/2} + y^{N/2} r^{N/4}$$
 for all cusps of $\Gamma_S$ and all $y \gg 1$.  By Theorem \ref{Upper-bd-thm}, it suffices to consider the range where $1 \ll y < r/2\pi$.  We consider the point $y_0 = r^{ (N/2 + 1 + 2\theta)/(N + 1 + 2\theta)}$, which is chosen so that the expressions $y^{N/2} r^{N/4}$ and $y^{-N/2 -1 - 2\theta} r^{3N/4 + 1 +2\theta}$ appearing in Theorem \ref{trace-thm} and Theorem \ref{Upper-bd-thm} are both equal to $r^{N/2 + N(1+2\theta) / 4(N + 1 + 2\theta) }$ (our desired upper bound) when evaluated at $y_0$.  An elementary computation gives us that $y_0 < r^{11/12}$.

The expression $r^{N/2} + y^{N/2} r^{N/4}$ is increasing in $y$, so for $1 \ll y \le y_0$ we have
\begin{align*}
|F(c_\infty^{-1}n_x a_y)| / \| F_f \|_2 &\ll r^{N/2} + y^{N/2} r^{N/4} 
\le r^{N/2} + y_0^{N/2} r^{N/4} \\
&\ll r^{N/2 + N(1+2\theta) / 4(N + 1 + 2\theta)}.
\end{align*}

We next suppose that $y_0 \le y < r / 2\pi$.  Because the upper bound given by Theorem \ref{Upper-bd-thm} is decreasing in $y$ (including across the transition point at $y = r^{11/12}$), we may obtain an upper bound for $|F(c_\infty^{-1}n_x a_y)| / \| F_f \|_2$ by evaluating the upper bound of Theorem \ref{Upper-bd-thm} at $y_0$.  As $y_0 < r^{11/12}$, this gives
\[
|F(c_\infty^{-1}n_x a_y)| / \| F_f \|_2 \ll y_0^{-N/2 -1 - 2\theta} r^{3N/4 + 1 +2\theta} = r^{N/2 + N(1+2\theta) / 4(N + 1 + 2\theta) }.
\]
Note that the Casimir eigenvalue $\Lambda \sim r^2$ since $r \gg 1$.  Combining the above computation completes the proof of the upper bound in Theorem \ref{supnorm-thm-intro}.

\end{document}